\documentclass[a4paper,12pt]{amsart} 
\pdfoutput=1

\title[Transverse realization of dense Mordell--Weil rank jumps]
{Analytically dense Mordell--Weil rank jumps on \\ elliptic surfaces via transverse realization}
\date{}
\author{Jun--Yong Park}


\usepackage[charter]{mathdesign}
\SetMathAlphabet{\mathcal}{normal}{OMS}{cmsy}{m}{n}

\DeclareFontShape{T1}{mdbch}{m}{up}{<->ssub*cmr/m/n}{}

\usepackage[dvipsnames]{xcolor}

\usepackage{amsmath,graphicx,amsthm,verbatim,mathtools,thmtools,tikz-cd,url}
\usepackage[a4paper,margin=3.16cm, headsep=15pt, headheight=2cm]{geometry}
\usepackage[shortlabels]{enumitem}
\usepackage[all,cmtip]{xy}
\usepackage{hyperref}
\usepackage{subfiles} 

\newtheorem{thm}{Theorem}[section]
\newtheorem{Mthm}[thm]{Main Theorem}
\newtheorem{lem}[thm]{Lemma}
\newtheorem{cor}[thm]{Corollary}
\newtheorem{prop}[thm]{Proposition}
\theoremstyle{definition}
\newtheorem{defn}[thm]{Definition}

\newtheorem{rmk}[thm]{Remark}

\newcommand{\bZ}{\mathbb{Z}}

\newcommand{\cX}{\mathcal{X}}

\newcommand{\cY}{\mathcal{Y}}

\newcommand{\cO}{\mathcal{O}}
\newcommand{\cD}{\mathcal{D}}

\newcommand{\cL}{\mathcal{L}}
\newcommand{\cW}{\mathcal{W}}

\newcommand{\cM}{\mathcal{M}}
\newcommand{\cS}{\mathcal{S}}
\newcommand{\cP}{\mathcal{P}}

\newcommand{\diff}{d}

\DeclareMathOperator{\rk}{rk}

\newcommand{\I}{{\mathop{\rm I}}}
\newcommand{\II}{{\mathop{\rm II}}}

\newcommand{\Ec}{\mathcal{E}}

\newcommand{\Wsb}{\cW_n^{\mathrm{sb}}}

\newcommand{\Pb}{\mathbb{P}}

\newcommand{\Cb}{\mathbb{C}}

\newcommand{\Fb}{\mathbb{F}}

\newcommand{\Qb}{\mathbb{Q}}
\newcommand{\Rb}{\mathbb{R}}

\DeclareMathOperator{\Spec}{Spec}

\DeclareMathOperator{\MW}{\mathrm{MW}}

\DeclareMathOperator{\NS}{NS}
\DeclareMathOperator{\Triv}{Triv}

\DeclareMathOperator{\PGL}{PGL}

\renewcommand{\setminus}{\smallsetminus}

\begin{document}
        
        \maketitle

    \vspace{-4ex}
    
    \begin{abstract}
    Let $\mathcal{W}_{\mathrm{min},n}$ be the moduli stack of minimal Weierstrass fibrations of Faltings height $n$ over an unparameterized $\mathbb{P}^1_{\mathbb{C}}$, and let $\mathcal{W}_n^{\mathrm{sb}}$ be its dense open substack of simply-branched Jacobian elliptic surfaces with only ${\mathop{\rm I}}_1$ singular fibres. For integers $n\ge2$ and $r$ satisfying
    \[
    1\le r\le\left\lfloor\frac{10n-2}{n-1}\right\rfloor
    \]
    we prove that every point of $\mathcal{W}_n^{\mathrm{sb}}$ is a limit of pairwise non-isomorphic surfaces in this locus carrying $r$ independent Mordell--Weil sections, while the least canonical height of a non-zero section tends to infinity. Thus the rank-at-least-$r$ locus is analytically dense in $\mathcal{W}_{\mathrm{min},n}$. First- and second-order infinitesimal criteria also give exact rank $r$ density for $1\le r\le10$ when $n\ge3$, and for $1\le r\le18$ when $n=2$. The direct realization argument combines Shepherd--Barron's infinitesimal period calculation with a block Vandermonde minor, lattice approximation, and the holomorphic implicit function theorem, and uses neither monodromy nor o-minimality. Finally, we construct section moduli stacks of finite type parametrizing all-${\mathop{\rm I}}_1$ Jacobian elliptic surfaces of Faltings height $n$, equipped with $r$ ordered non-zero sections of fixed canonical heights. At the constructed surfaces and sections, these algebraic stacks are smooth of dimension $10n-2-r(n-1)$, and their germs after base change to marked deformation charts identify with simultaneous fixed class Hodge locus germs cut out transversely by the period equations.
    \end{abstract}


    \section{Introduction}
    \label{sec:Intro}

    The global sections of a Jacobian elliptic surface form its Mordell--Weil group. At each fixed Faltings height $n\ge2$, a very general such surface over $\Pb^1_{\Cb}$ has rank zero. For $n\ge3$, this follows from Cox's result on the very general Picard number \cite[Prop.~3.2]{Cox_NL}; the elliptic $K3$ case $n=2$ is classical. We study how additional sections appear under arbitrarily small complex deformations at the same height. We establish analytic density of Jacobian elliptic surfaces with Mordell--Weil rank at least $r$ for every positive integer $r$ in an explicit transverse range, and with rank exactly $r$ for $1\le r\le10$ when $n\ge3$ and $1\le r\le18$ when $n=2$. Using Shepherd-Barron's lucid description of the infinitesimal period map \cite{SB_Torelli}, we coordinate local branch deformations through global period equations to establish transverse intersections of the Jacobian elliptic surface period image with simultaneous Hodge loci. The resulting surfaces carry independent sections, while the least canonical height of a non-zero section tends to infinity. We also construct algebraic section moduli stacks of finite type parametrizing surfaces with ordered non-zero sections of fixed canonical heights. At the constructed tuples, their germs after base change to marked deformation charts identify with the corresponding smooth simultaneous fixed class Hodge locus germs.

    \medskip

    Throughout the paper we work over $\Cb$. By a Jacobian elliptic surface over $\Pb^1$ we mean a smooth projective surface $S$ equipped with a morphism $f\colon S\to\Pb^1$ with connected fibres and smooth generic fibre of genus one, together with a distinguished section $\sigma\colon\Pb^1\to S$, called the zero section. We assume that $f$ is relatively minimal, meaning that no fibre contains a $(-1)$-curve. Write $\omega_{S/\Pb^1}$ for the relative dualizing sheaf. Its fundamental line bundle and Faltings height are
    \[
    \cL=(R^1f_*\cO_S)^\vee\simeq f_*\omega_{S/\Pb^1},
    \qquad
    n=\deg\cL=\chi(\cO_S)
    \]
    We assume $n\ge2$. The minimal discriminant divisor has degree $12n$; see \cite{Miranda2,BHPV,SS} for background on elliptic surfaces and their invariants.

    \medskip

    Let $\cW_{\min,n}$ denote the moduli stack of minimal Weierstrass fibrations of height $n$ over an unparameterized $\Pb^1$. It is obtained from the fixed-base height moduli stack of minimal Weierstrass fibrations with smooth generic fibre by quotienting by changes of coordinates on the base \cite[Thms.~1.1--1.2]{JJ}. Since the base $\Pb^1$ is unparameterized, isomorphisms may change its coordinates but must carry fibres to fibres and preserve the zero section. The fixed-base height moduli description by minimal weighted linear series is given in \cite[Prop.~7.5]{BPS}. We identify the geometric points of $\cW_{\min,n}$ with isomorphism classes of relatively minimal Jacobian elliptic surfaces by taking the minimal resolution of each Weierstrass model; conversely, the Weierstrass model is obtained by contracting the fibre components disjoint from the zero section \cite[Defs.~7.2--7.3]{BPS}. Analytic density is understood in the usual complex analytic topology on the coarse moduli space; equivalently, it can be checked after pullback to an analytic atlas.

    \medskip

    After choosing homogeneous coordinates $[u:v]$ on the base and setting $z=u/v$, write $E_S/\Cb(z)$ for the generic fibre and
    \[
    \MW(S/\Pb^1)=E_S(\Cb(z))
    \]
    for the Mordell--Weil group of sections. Let $F$ and $O$ denote the fibre and zero section classes, and let $R(S)$ be the root lattice generated by the components of reducible fibres that do not meet $O$. Let $\NS(S)$ denote the N\'eron--Severi group of $S$. The trivial lattice $\Triv(S)\subset\NS(S)$ is generated by $F$, $O$, and these components. The Shioda--Tate description \cite[Thm.~1.3]{Shioda} gives
    \[
    0\longrightarrow
    \Triv(S)
    \longrightarrow
    \NS(S)
    \longrightarrow
    \MW(S/\Pb^1)
    \longrightarrow0
    \]
    Writing $\rho(S)=\rk\NS(S)$, we obtain
    \begin{equation}\label{eq:ST}
    \rho(S)=2+\rk R(S)+\rk\MW(S/\Pb^1)
    \end{equation}

    The root lattice term is determined by the Kodaira types of the singular fibres. To remove this contribution, we work on the squarefree discriminant locus $\cW_n^{\I_1}\subset\cW_{\min,n}$. Every singular fibre on this locus is of Kodaira type $\I_1$, that is, an irreducible rational curve with a single ordinary node. Thus $R(S)=0$, and \eqref{eq:ST} becomes
    \[
    \rk\MW(S/\Pb^1)=\rho(S)-2
    \]
    On this locus, algebraic classes that are linearly independent modulo $\langle F,O\rangle$ give linearly independent Mordell--Weil sections.

    \medskip

    To describe these classes, put
    \[
    \Lambda_n=\langle F,O\rangle^\perp\subset H^2(S,\bZ)/\mathrm{tors}
    \]
    Its isomorphism type depends only on $n$; we use the same symbol for a fixed abstract model. Throughout the paper, integral cohomology lattices are understood modulo torsion. For a lattice or finitely generated abelian group $A$, write $A_{\mathbf K}=A\otimes_{\bZ}\mathbf K$ for $\mathbf K\in\{\Qb,\Rb,\Cb\}$. For primitive $(1,1)$-cohomology, we write
    \[
    \begin{aligned}
    H^{1,1}_{\mathrm{prim}}(S)
    &=H^{1,1}(S)\cap\Lambda_{n,\Cb}\\
    H^{1,1}_{\mathrm{prim}}(S,\Rb)
    &=H^{1,1}(S)\cap\Lambda_{n,\Rb}
    \end{aligned}
    \]
    We use the half-normalized N\'eron--Tate pairing
    \[
    \langle P,Q\rangle_{\widehat h}
    =\frac12\langle P,Q\rangle_{\mathrm{Sh}},
    \qquad
    \widehat h(P)=\langle P,P\rangle_{\widehat h}
    \]
    where $\langle\, ,\,\rangle_{\mathrm{Sh}}$ is Shioda's lattice height pairing. On the all-$\I_1$ locus, a non-zero section $P$ satisfies $\widehat h(P)=n+P\cdot O$, where $P\cdot O$ counts its intersections with the zero section, with multiplicity; see Proposition~\ref{prop:two-pointed-height}. On this locus, we call the orthogonal projection of the section class $[P]$ to $\Lambda_n$ its \emph{Shioda class}; explicitly, $\gamma_P=[P]-[O]-(n+P\cdot O)[F]\in\Lambda_n$. The Shioda map identifies the Mordell--Weil group with the primitive integral Hodge lattice on the all-$\I_1$ locus:
    \[
    \sigma_{\mathrm{Sh}}\colon
    \MW(S/\Pb^1)\xrightarrow{\sim}\Lambda_n\cap H^{1,1}(S),
    \qquad
    P\longmapsto\gamma_P
    \]
    For this integral identification, see the proof of Theorem~\ref{thm:KS-higher-rank}.

    \medskip

    The moduli stack $\cW_n^{\I_1}$ is irreducible of dimension $10n-2$, as follows from the quotient description in \cite[Thms.~1.1 and~1.2(a)]{JJ}. Each Jacobian elliptic surface in this locus has a classifying morphism \cite[Prop.~9]{HP}
    \[
    \varphi_S\colon\Pb^1\longrightarrow\overline{\cM}_{1,1}
    \]
    Here $\overline{\cM}_{1,1}$ is the moduli stack of stable genus one curves with one marked point, and $\cM_{1,1}$ is its open substack parametrizing smooth curves. This morphism is unramified over the boundary $\overline{\cM}_{1,1}\setminus\cM_{1,1}$; see the local calculation in the proof of \cite[Prop.~11]{HP}. Its stack theoretic ramification divisor $\operatorname{Ram}(\varphi_S)$ is the zero divisor of its differential; it differs from the ramification divisor of the coarse $j$-map. We call $S$ \emph{simply-branched} when $\operatorname{Ram}(\varphi_S)$ is reduced, and denote the resulting substack by
    \[
    \Wsb\subset\cW_n^{\I_1}
    \]
    Distinct ramification points are allowed to have the same image under $\varphi_S$. Lemmas~\ref{lem:sb-open-dense} and~\ref{lem:sb-dense-minimal} show that $\Wsb$ is Zariski open and dense in $\cW_{\min,n}$, hence also analytically dense.

    \medskip

    Elliptic $K3$ surfaces with $n=2$ have Kodaira dimension $\kappa = 0$. They provide the archetype for our question. Since $F\cdot F=0$, $F\cdot O=1$, and $O^2=-2$, the classes $F,O+F$ identify $\langle F,O\rangle$ with the hyperbolic plane $U$. We recall a standard consequence of local Torelli and the Shioda--Tate formula, including its proof to explain the comparison with surfaces of height $n\ge3$, which have Kodaira dimension $\kappa = 1$.

    \begin{prop}\label{prop:k3-archetype}
    The locus
    \[
    \bigl\{
    [S]\in\cW_2^{\I_1}:
    \rk\MW(S/\Pb^1)=18
    \bigr\}
    \]
    is analytically dense in $\cW_2^{\I_1}$. Consequently, for every $1\le r\le18$, the locus of Jacobian elliptic surfaces with Mordell--Weil rank at least $r$ is analytically dense in $\cW_2^{\I_1}$.
    \end{prop}

    \begin{proof}
    Fix $[S]\in\cW_2^{\I_1}$. The lattice $\Lambda_2=\langle F,O\rangle^\perp$ has signature $(2,18)$. By local Torelli, the marked deformations of $S$ on which $F$ and $O$ remain of type $(1,1)$ are locally parametrized by an open subset of a connected component $\cD_2^+$ of
    \[
    \left\{
    \Pi\subset\Lambda_{2,\Rb}:
    \Pi\text{ is an oriented positive two-plane}
    \right\}
    \]
    see \cite[Ch.~6, Prop.~2.8 and \S2.4]{Huybrechts_K3}. After shrinking to a small analytic ball, the relative exponential sequence extends $\cO_S(F)$ and $\cO_S(O)$ because their classes remain of type $(1,1)$. On the central surface, $\cO_S(F)$ and $\cO_S(O)$ have spaces of sections of dimensions two and one, respectively, and vanishing higher cohomology. Cohomology and base change therefore extend the pencil and its zero section divisor. After shrinking, the pencil remains basepoint free and the divisor remains a smooth rational curve; its intersection number with a fibre is one, so it defines a section. The extended class $O+3F$ has square $4$, so the projectivity criterion for $K3$ surfaces shows that the nearby surfaces are projective. The squarefree discriminant condition is open. Thus we obtain a marked chart $B$ in $\cW_2^{\I_1}$ whose period image is open in $\cD_2^+$.

    Positive two-planes defined over $\Qb$ are dense in $\cD_2^+$. Indeed, an oriented basis of any positive real two-plane can be approximated by vectors in $\Lambda_{2,\Qb}$, and sufficiently close approximations still span positive planes with the induced orientation. Choose a rational plane $\Pi$ in the period image of $B$, and let $S_\Pi$ be the corresponding surface.

    Since $\Pi$ is defined over $\Qb$, its orthogonal complement is a rational subspace of dimension $18$. Hence
    \[
    \rk\bigl(\Pi^\perp\cap\Lambda_2\bigr)=18
    \]
    These integral classes are of type $(1,1)$ on $S_\Pi$, so the Lefschetz $(1,1)$-theorem makes them algebraic. Together with $F$ and $O$, these give $20$ independent classes in $\NS(S_\Pi)$. Since divisor classes lie in $H^{1,1}(S_\Pi)$, we have
    \[
    20\le\rho(S_\Pi)\le h^{1,1}(S_\Pi)=20
    \]
    Hence $\rho(S_\Pi)=20$.

    The surface $S_\Pi$ lies in $\cW_2^{\I_1}$, so its singular fibres are irreducible. Thus \eqref{eq:ST} gives
    \[
    \rk\MW(S_\Pi/\Pb^1)=\rho(S_\Pi)-2=18
    \]
    Rational positive planes occur in every open subset of the local period image, proving density. The assertion for rank at least $r$ follows immediately.
    \end{proof}

    This is the classical period theoretic density argument with the fibre and zero section classes fixed. For the density of $K3$ surfaces of Picard number $20$ in polarized moduli, see \cite[\S4, p.~401]{HulekKloosterman_CubicFourfold}; compare the period domain argument in \cite[proof of Cor.~VIII.8.5]{BHPV}. For related results, Oguiso proves density of every possible Picard number in the Kuranishi space and constructs a hierarchy of narrow Mordell--Weil lattices of Jacobian $K3$ surfaces \cite[Cor.~1.3 and Thm.~1.4]{Oguiso_LocalK3}.

    \medskip

    In the $K3$ case, local Torelli identifies the marked deformation space preserving $F$ and $O$ with an open subset of $\cD_2^+$. At higher height, the dimension of the ambient period domain prevents such local openness. The primitive Hodge numbers are
    \[
    h^{2,0}=h^{0,2}=n-1,
    \qquad
    h^{1,1}_{\mathrm{prim}}=10n-2
    \]
    For general $n$, let $\cD_n^+$ denote the connected component of the corresponding polarized weight-$2$ period domain containing the marked periods under consideration. For $n\ge3$, its dimension satisfies
    \[
    \dim_{\Cb}\cD_n^+
    =(n-1)(10n-2)+\binom{n-1}{2}
    >10n-2
    =\dim\cW_n^{\I_1}
    \]
    The local elliptic surface period image therefore has positive codimension in $\cD_n^+$. Density of Hodge loci in the ambient domain alone does not guarantee that they meet this image. A fixed class $\gamma\in\Lambda_n$ with $\gamma^2<0$ defines a Hodge locus of complex codimension $n-1$ in $\cD_n^+$. This ambient codimension does not by itself imply that the period equations have independent differentials on the deformation space.

    \medskip

    For an ordered tuple $\mathbf d=(d_1,\ldots,d_r)$ of positive integers, let $\cS_{n,\mathbf d}^{\I_1}$ denote the section moduli stack of height $n$ all-$\I_1$ Jacobian elliptic surfaces equipped with ordered non-zero sections $P_1,\ldots,P_r$ satisfying $\widehat h(P_a)=d_a$. The base is considered up to isomorphism. Subsection~\ref{subsec:two-pointed-moduli} constructs this algebraic Deligne--Mumford stack of finite type and its representable forgetful morphism
    \[
    \rho\colon\cS_{n,\mathbf d}^{\I_1}\longrightarrow\cW_n^{\I_1}
    \]
    The canonical heights are locally constant in families with chosen sections, so fixing $\mathbf d$ does not change the deformation germ of a tuple.

    \medskip

    Our main theorem constructs independent Mordell--Weil sections through transverse period conditions.

    \begin{Mthm}\label{thm:intro-realization}
    Let $n\ge2$ and $r\ge1$ be integers satisfying
    \[
    r(n-1)\le10n-2
    \]
    Let $\cW_{\min,n}$ denote the complex moduli stack of minimal Weierstrass fibrations of Faltings height $n$ over an unparameterized $\Pb^1_{\Cb}$. Let $\Wsb\subset\cW_{\min,n}$ be the locus where every singular fibre has Kodaira type $\I_1$ and the differential of the classifying morphism
    \[
    \varphi_S\colon\Pb^1_{\Cb}\longrightarrow\overline{\cM}_{1,1}
    \]
    has reduced zero divisor. Every $[S]\in\Wsb$ is an analytic limit of pairwise non-isomorphic surfaces $[S_m]\in\Wsb$ carrying $r$ independent Mordell--Weil sections $P_{1,m},\ldots,P_{r,m}$ with the following properties.
    \begin{enumerate}[\normalfont(1)]
    \item The Shioda classes $\gamma_{1,m},\ldots,\gamma_{r,m}\in\Lambda_n$ of the chosen sections are indivisible, and
    \[
    \min_{0\ne P\in\MW(S_m/\Pb^1)}
    \widehat h(P)\longrightarrow\infty
    \qquad\text{as }m\to\infty
    \]
    \item On a marked local deformation chart, the $r(n-1)$ simultaneous fixed class period equations for these classes have linearly independent differentials at $S_m$. Their common zero locus is therefore a smooth simultaneous fixed class Hodge locus of complex codimension $r(n-1)$ near $S_m$.
    \item Put $\mathbf d_m=(\widehat h(P_{1,m}),\ldots,\widehat h(P_{r,m}))$. The stack $\cS_{n,\mathbf d_m}^{\I_1}$ is smooth at $(S_m,P_{1,m},\ldots,P_{r,m})$ of complex dimension $10n-2-r(n-1)$. After base change to the chart in \textup{(2)}, the forgetful morphism identifies the section moduli germ at the chosen lift with the Hodge germ in \textup{(2)}. On a sufficiently small connected representative of this Hodge germ, the chosen sections extend uniquely and their height pairing matrix is constant.
    \end{enumerate}
    \end{Mthm}

    Consequently, the rank-at-least-$r$ locus in $\Wsb$ is analytically dense in $\cW_{\min,n}$. Taking $r=10$ gives $8$ dimensional simultaneous Hodge germs for every $n\ge2$. Their representatives can be chosen with images in any prescribed non-empty analytic open subset and with arbitrarily large uniform lower bounds for all non-zero section heights; see Corollary~\ref{cor:rank-density}.

    \medskip

    We record the period equations on a single marked deformation chart.

    \begin{thm}\label{thm:intro-transverse}
    The surfaces and sections in Theorem~\ref{thm:intro-realization} may be chosen in a single simply-connected analytic \'etale local deformation chart $B\to\Wsb$ of dimension $10n-2$, with $S_{b_0}=S$, $S_{b_m}=S_m$, and $b_m\to b_0$. A marking identifies the lattices
    \[
    \langle F_b,O_b\rangle^\perp
    \subset H^2(S_b,\bZ)/\mathrm{tors}
    \]
    with a fixed lattice $\Lambda_n$ by parallel transport preserving the cup product pairing, where $F_b,O_b$ are the fibre and zero section classes.

    After shrinking $B$ and discarding finitely many initial terms, choose a holomorphic frame $\omega_0,\ldots,\omega_{n-2}$ of the Hodge bundle $F^2$ with fibre $H^{2,0}(S_b)$. For the Shioda classes $\gamma_{1,m},\ldots,\gamma_{r,m}$ of the chosen sections, define
    \[
    \Phi_m\colon B\longrightarrow\Cb^{r(n-1)}
    \qquad
    \Phi_m(b)=\bigl(\gamma_{a,m}\cdot\omega_j(b)\bigr)_{\substack{1\le a\le r\\0\le j\le n-2}}
    \]
    Then
    \[
    \Phi_m(b_m)=0
    \qquad
    \rk_{\Cb}(d\Phi_m)_{b_m}=r(n-1)
    \]
    After base change to $B$, forgetting the chosen ordered sections identifies the section moduli germ at the chosen lift of $b_m$ with the Hodge germ $(\Phi_m^{-1}(0),b_m)$.
    \end{thm}

    Schiffer variations arise from elementary local changes of the holomorphic gluing of a Riemann surface \cite[\S3]{Schiffer53}. On the simply-branched locus, Shepherd--Barron's adaptation to elliptic surfaces provides local deformation coordinates indexed by the ramification points of the classifying morphism \cite[Prop.~4.4]{SB_Torelli}. Their tangent directions form a basis of the deformation tangent space. Each direction induces a rank-one map from $H^{2,0}(S)$ to $H^{1,1}_{\mathrm{prim}}(S)$ \cite[Thm.~1.4(3)]{SB_Torelli}. We call non-zero generators of these image lines the \emph{Schiffer image classes}; they form an orthogonal basis of primitive $(1,1)$-cohomology \cite[Thm.~4.11(4)]{SB_Torelli}. Evaluation of holomorphic two-forms at distinct ramification points gives Vandermonde minors. Our block construction uses these to show that a maximal Jacobian minor of the simultaneous period equations is not identically zero.

    We choose a real tuple in $H^{1,1}_{\mathrm{prim}}(S,\Rb)$ with negative-definite Gram matrix at which the chosen minor does not vanish. We approximate its directions by those of indivisible lattice vectors whose norms tend to infinity in a fixed auxiliary Euclidean norm. Applying the holomorphic implicit function theorem, with the normalized lattice vectors as parameters, gives nearby surfaces on which all their period pairings vanish. Rescaling to the original integral vectors preserves these vanishing pairings. Since the vectors are real, they are also orthogonal to $H^{0,2}$ and hence have type $(1,1)$.

    Independence persists under sufficiently close approximation. Moreover, the normalized limiting vectors have negative square, so $-\gamma_{a,m}^2\to+\infty$ for every $a$. By the Lefschetz $(1,1)$-theorem, these classes are algebraic. On the all-$\I_1$ locus, they correspond under Shioda's identification to independent sections satisfying
    \[
    \widehat h(P_{a,m})
    =-\frac12\gamma_{a,m}^2
    \longrightarrow+\infty
    \]
    These heights are integers because $\Lambda_n$ is even. To control the heights of all non-zero sections, we use Lemma~\ref{lem:bounded-height-avoidance}. On a relatively compact marked chart, only finitely many non-zero integral classes can become Hodge with $-\gamma^2/2\le H$. Their proper closed Hodge loci have an open dense complement. Applying the local realization theorem in these complements, with increasing $H$ and shrinking neighbourhoods, makes the least canonical height of a non-zero section tend to infinity.

    \medskip

    To obtain exact rank $1$, we use an explicit infinitesimal criterion. Suppose that a non-zero class $\alpha\in\Lambda_n$ is indivisible, is of type $(1,1)$ at a point of $\Wsb$, and pairs non-trivially with every Schiffer image class. For each $\beta\in\Lambda_n\setminus\Qb\alpha$ that is Hodge at that point, some first-order deformation preserving the Hodge type of $\alpha$ does not preserve that of $\beta$. Consequently, outside a countable union of proper closed analytic subsets, a sufficiently small connected representative of the Hodge germ defined by $\alpha$ parametrizes surfaces of Mordell--Weil rank exactly $1$; see Proposition~\ref{prop:er1-generic}. This criterion and its simultaneous version, combined with rational approximation and local realization, give exact-rank density for $r\le9$ when $n\ge3$ and for $r\le18$ when $n=2$. The second-order argument of Subsection~\ref{subsec:second-order-rank-ten} supplies the remaining case $r=10$ when $n\ge3$. Together these arguments give the following result, proved in Theorem~\ref{thm:er-density}.

    \begin{cor}\label{cor:intro-exact-rank}
    Let $n\ge2$ and $r\ge1$ be integers, with $r\le18$ if $n=2$ and $r\le10$ if $n\ge3$. Then the locus of surfaces in $\Wsb$ with Mordell--Weil rank exactly $r$ is analytically dense in $\cW_{\min,n}$.
    \end{cor}

    These first- and second-order arguments use neither monodromy nor o-minimality.

    \medskip

    The constant dimension $8$ also has a moduli theoretic interpretation. $10$ additional sections together with the zero section give $11$ markings on the generic elliptic fibre. The corresponding fibre product target, constructed in Subsection~\ref{subsec:two-pointed-moduli}, has dimension $11$ and trivial canonical bundle on the open locus containing the section maps. The expected deformation dimension for these maps from $\Pb^1$, modulo source automorphisms, is therefore $11-3=8$. Remark~\ref{rmk:ten-section-canonical} explains this calculation and its relation to the Kodaira dimensions of $\overline M_{1,r+1}$.

    \medskip

    Cox and Kloosterman studied the Noether--Lefschetz and Mordell--Weil rank loci of Jacobian elliptic surfaces. For $n\ge3$, Cox proves that the positive Mordell--Weil rank locus has dimension $9n-1$ \cite{Cox_NL} and that the union of the codimension-$(n-1)$ components of the Noether--Lefschetz locus is dense in the classical topology \cite[Thm.~3.1(iii)]{Cox_NL}. For every $n\ge2$ and every $H>0$, he also proves density of surfaces with positive Mordell--Weil rank for which every non-zero section $P$ satisfies $P\cdot O\ge H$ \cite[Thm.~4.4]{Cox_NL}. On the all-$\I_1$ locus, our normalization satisfies $\widehat h(P)=n+P\cdot O$, so this also gives arbitrarily large lower bounds for all non-zero section heights. Kloosterman determines the dimensions of higher Noether--Lefschetz loci on the non-constant $j$ locus and obtains upper bounds for the dimensions of Mordell--Weil rank loci there \cite[Thm.~1.1 and Cor.~1.3]{Kloosterman_NL}.

    These results fit into the general study of density and distribution of Hodge loci for variations of Hodge structure. Khelifa--Urbanik give density criteria under monodromy and expected dimension hypotheses \cite{Khelifa_Urbanik}. Under suitable hypotheses on a weight-two variation, Tayou--Tholozan prove equidistribution of transverse refined Noether--Lefschetz points, ordered by lattice discriminant, when the expected dimension is zero \cite[Thm.~6.10]{Tayou_Tholozan}.

    \begin{rmk}\label{rmk:KS-rank-bound}
    Higher Mordell--Weil ranks occur in classical Delsarte examples with cyclic symmetry; see \cite{Heijne} for a classification and rank computations. Shioda's isotrivial example has rank $68$ at height $60$ \cite[pp.~109--110]{Shioda92}. The non-isotrivial example studied by Shioda \cite{Shioda86} and Stiller \cite[Thm.~4.2]{Stiller87} has rank $56$ at height $210$. The former has only type-$\mathrm{II}$ singular fibres. The latter admits the model
    \[
    y^2=x^3+x+z^{1260}
    \]
    whose singular fibres are all of type $\I_1$, but whose stack theoretic ramification divisor is
    \[
    \operatorname{Ram}(\varphi_S)=1259[0]+839[\infty]
    \]
    Thus neither example lies in $\Wsb$.
    \end{rmk}

    \begin{defn}\label{def:rank-jumping-loci}
    For $n\ge3$, let $Z\subset B$ be a connected smooth $8$ dimensional simultaneous fixed class Hodge locus in a marked chart $B\to\Wsb$ parametrizing surfaces $S_b$, cut out by $10$ independent rational classes as in Proposition~\ref{prop:second-order-rank-ten}. Every $b\in Z$ satisfies $\rk\MW(S_b/\Pb^1)\ge10$, and the \emph{Mordell--Weil rank jumping locus} of $Z$ is
    \[
    \mathcal E_Z
    \coloneqq
    \bigl\{b\in Z:\rk\MW(S_b/\Pb^1)\ge11\bigr\}
    \]
    \end{defn}

    The locus $\mathcal E_Z$ is a countable union of proper closed analytic subsets of $Z$, and its dense complement $Z\setminus\mathcal E_Z$ is precisely the rank $10$ locus.

    \medskip

    Whether Mordell--Weil ranks of positive-height Jacobian elliptic surfaces over $\Pb^1_{\Cb}$ are unbounded as the height grows remains open \cite[Question~14.3]{SS}. Unbounded ranks on the simply-branched loci $\Wsb$ as $n$ varies would suffice, but our results do not reduce this question to $\Wsb$ or to the loci $\mathcal E_Z$. Over each fixed finite field $\Fb_q$, Mordell--Weil ranks over $\Fb_q(t)$ are unbounded: Tate--Shafarevich established this for isotrivial curves in odd characteristic \cite{TS67}, and Ulmer constructed non-isotrivial examples in every positive characteristic \cite[Thm.~1.5]{Ulmer}.

    \subsection{Outline of the paper}

    Section~\ref{sec:primitive-cohomology} records the primitive lattice and the period conditions used in the proof. Section~\ref{sec:Transversality} establishes simultaneous transverse realization, the height bounds, and analytic density. Subsections~\ref{subsec:er1} and~\ref{subsec:second-order-rank-ten} establish exact-rank density in the ranges of Theorem~\ref{thm:er-density} using first- and second-order period equations. Finally, Subsection~\ref{subsec:two-pointed-moduli} constructs the section moduli stacks, identifies canonical heights with two-pointed degrees, and proves the section moduli smoothness and germ comparison, completing the proofs of Theorems~\ref{thm:intro-realization} and~\ref{thm:intro-transverse}.

    \medskip

    \section{Primitive cohomology and period conditions}
    \label{sec:primitive-cohomology}

    We recall the lattice and Hodge theoretic facts used in the realization argument. Our general reference for Hodge theory, period domains, and variations of Hodge structure throughout the paper is Carlson--M\"uller-Stach--Peters \cite{CMSP}. We use a dot for the intersection pairing and its complex bilinear extension.

    \subsection{The primitive lattice}

    Let $S$ be a smooth complex projective surface, and let $f\colon S\to\Pb^1$ be a relatively minimal Jacobian elliptic fibration with $\chi(\cO_S)=n\ge2$. Let $F$ and $O$ denote a fibre divisor and the zero section, respectively; we use the same symbols for their divisor classes. The canonical bundle formula gives $K_S=(n-2)F$. Since $F\cdot F=0$ and $F\cdot O=1$, adjunction gives $O^2=-n$. The standard invariant formulas
    \[
    q(S)=0,\qquad p_g(S)=n-1,\qquad e(S)=12n
    \]
    give $b_2(S)=12n-2$; see \cite{Miranda2,SS}. Hodge decomposition and the Hodge index theorem give intersection signature $(2n-1,10n-1)$ on $H^2(S,\bZ)/\mathrm{tors}$.

    The fibre and zero section classes have intersection matrix
    \[
    \begin{pmatrix}
    0&1\\
    1&-n
    \end{pmatrix}
    \]
    Its determinant is $-1$, so $\langle F,O\rangle$ is a primitive unimodular sublattice of signature $(1,1)$. Its orthogonal complement
    \[
    \Lambda_n=\langle F,O\rangle^\perp\subset H^2(S,\bZ)/\mathrm{tors}
    \]
    therefore has rank $12n-4$ and signature $(2n-2,10n-2)$. Orthogonal projection gives an integral decomposition
    \[
    H^2(S,\bZ)/\mathrm{tors}
    =
    \langle F,O\rangle\oplus\Lambda_n
    \]

    We recall the standard description of the primitive lattice and its proof \cite[Thm.~24]{Greer}. The lattice $\Lambda_n$ is even. Indeed, the Wu formula gives
    \[
    \gamma^2\equiv\gamma\cdot K_S=(n-2)\gamma\cdot F=0\pmod2
    \qquad(\gamma\in\Lambda_n)
    \]
    Poincar\'e duality makes $H^2(S,\bZ)/\mathrm{tors}$ unimodular, so the orthogonal decomposition also shows that $\Lambda_n$ is unimodular. The classification of indefinite even unimodular lattices gives
    \begin{equation}\label{eq:primitive-lattice}
    \Lambda_n\simeq U^{\oplus(2n-2)}\oplus E_8(-1)^{\oplus n}
    \end{equation}
    Here $U$ is the even hyperbolic plane, with Gram matrix $\left(\begin{smallmatrix}0&1\\1&0\end{smallmatrix}\right)$, and $E_8(-1)$ is the negative-definite $E_8$ lattice; see \cite[Ch.~V, Thm.~5]{Serre_Arithmetic}. We use $\Lambda_n$ also for a fixed abstract model of this lattice.

    The primitive Hodge numbers are $h^{2,0}=h^{0,2}=p_g(S)=n-1$ and $h^{1,1}_{\mathrm{prim}}=10n-2$. The real subspace
    \[
    \bigl(H^{2,0}(S)\oplus H^{0,2}(S)\bigr)\cap\Lambda_{n,\Rb}
    \]
    is positive definite for the intersection pairing. Since $H^{1,1}(S)\cap H^2(S,\Rb)$ has signature $(1,10n-1)$ by the Hodge index theorem and $\langle F,O\rangle_{\Rb}$ has signature $(1,1)$, its orthogonal complement $H^{1,1}_{\mathrm{prim}}(S,\Rb)$ is negative definite. In particular, every non-zero class $\gamma\in\Lambda_n\cap H^{1,1}(S)$ satisfies
    \[
    -\frac12\gamma^2\in\bZ_{>0}
    \]
    Here primitive cohomology means the orthogonal complement of $F$ and $O$. To avoid a clash with this usage, we call a lattice vector \emph{indivisible} if it is non-zero and not of the form $k\gamma'$ with $\gamma'$ in the lattice and $k\ge2$ an integer.

    \subsection{Periods and fixed class Hodge loci}\label{subsec:periods}

    A primitive Hodge structure is determined by its space of holomorphic two-forms. Let $\cD_n^+$ be a connected component of the period domain parametrizing $(n-1)$ dimensional subspaces $F^2\subset\Lambda_{n,\Cb}$ satisfying
    \[
    F^2\cdot F^2=0,
    \qquad
    \omega\cdot\overline{\omega}>0
    \quad\text{for every }0\ne\omega\in F^2
    \]
    Its complex dimension is $(n-1)(10n-2)+\binom{n-1}{2}$. It is a type-IV Hermitian symmetric domain when $n=2$; for $n\ge3$, it is a Griffiths period domain that is not Hermitian symmetric. For this description of the period domain for elliptic surfaces, see \cite[\S3, p.~2345]{Greer}; for the general Hodge theoretic background, see \cite[Ch.~4]{CMSP}.

    For a real class $\gamma\in\Lambda_{n,\Rb}$, the condition
    \[
    \gamma\perp F^2
    \]
    is equivalent to being of type $(1,1)$. Indeed, reality also gives orthogonality to $\overline{F^2}$, and the orthogonal complement of $F^2\oplus\overline{F^2}$ is the primitive $(1,1)$-subspace.

    If $\gamma^2<0$, its fixed class Hodge locus in $\cD_n^+$ is smooth of complex codimension $n-1$. Requiring $F^2\subset\gamma^\perp$ removes one negative direction from the intersection space, replacing $10n-2$ by $10n-3$ in the dimension formula above. More generally, a fixed negative-definite span of dimension $r$ defines a smooth ambient Hodge locus of complex codimension $r(n-1)$.

    To express these conditions in a family, let $B$ be a small analytic ball carrying a smooth projective family of height $n$ Jacobian elliptic surfaces with zero section. The fibre and zero section classes remain flat and algebraic. Choose a marking at a base point $b_0\in B$, that is, an isometry from the primitive cohomology lattice of $S_{b_0}$ to $\Lambda_n$. Since $B$ is simply-connected, parallel transport extends this to compatible identifications of the primitive cohomology lattices of all fibres with $\Lambda_n$. Choose a local holomorphic frame $\omega_0(b),\ldots,\omega_{n-2}(b)$ of the Hodge bundle $F^2$. For fixed classes $\gamma_1,\ldots,\gamma_r\in\Lambda_{n,\Rb}$, their simultaneous Hodge locus is defined by the holomorphic period equations
    \[
    \gamma_a\cdot\omega_j(b)=0,
    \qquad
    1\le a\le r,\quad 0\le j\le n-2
    \]

    Let $\mathrm{per}\colon B\to\cD_n^+$ be the marked period map, and let $\cD_{\boldsymbol\gamma}\subset\cD_n^+$ be the locus where all the classes $\gamma_a$ have type $(1,1)$. Their simultaneous Hodge locus $Z_{\boldsymbol\gamma}$ in $B$ fits into the Cartesian square
    \[
    \begin{tikzcd}[column sep=large, row sep=large]
    Z_{\boldsymbol\gamma} \arrow[r, hook] \arrow[d]
    &
    B \arrow[d, "\mathrm{per}"]
    \\
    \cD_{\boldsymbol\gamma} \arrow[r, hook]
    &
    \cD_n^+
    \end{tikzcd}
    \]

    The $r(n-1)$ period equations defining $Z_{\boldsymbol\gamma}$ need not have independent differentials on a Jacobian elliptic surface deformation space. Section~\ref{sec:Transversality} constructs suitable integral classes and nearby surfaces at which all these differentials are independent.

    \medskip

    \section{Geometric realization and density of Mordell--Weil rank jumps}
    \label{sec:Transversality}

    We construct transverse intersections of the local period image of $\Wsb$ with simultaneous fixed class Hodge loci in $\cD_n^+$. We then identify the resulting integral Hodge classes with independent Mordell--Weil sections. Fix a height $n\ge2$, and put $h\coloneqq n-1$ and $N\coloneqq10n-2$.

    The all-$\I_1$ locus $\cW_n^{\I_1}$ is obtained from a non-empty Zariski open subset of the coefficient space
    \[
    V_n\coloneqq H^0(\Pb^1,\cO(4n))\oplus H^0(\Pb^1,\cO(6n))
    \]
    by quotienting by Weierstrass scaling and changes of coordinates on the base \cite[Thms.~1.1 and~1.2(a)]{JJ}. The coefficient space is irreducible, so its non-empty Zariski open subset and the quotient $\cW_n^{\I_1}$ are irreducible. Moreover,
    \[
    \dim\cW_n^{\I_1}=(4n+1)+(6n+1)-1-3=10n-2=N
    \]

    We use the primitive lattice and period equations of Section~\ref{sec:primitive-cohomology}. The construction applies to every positive integer $r$ satisfying $rh\le N$; neither the integral classes nor their Gram matrices are prescribed in advance.

    \subsection{The simply-branched locus}
    \label{sec:simply_branched}

    We first identify a dense open locus on which each ramification point supplies one deformation parameter. Let
    \[
    \varphi_S\colon\Pb^1\longrightarrow\overline{\cM}_{1,1}\simeq\cP(4,6)
    \]
    be the classifying morphism of $[S]\in\cW_n^{\I_1}$, where $\cP(4,6)$ is the weighted projective stack. Following \cite[\S2]{SB_Torelli}, we write $\operatorname{Ram}(\varphi_S)$ for the zero divisor of the stack theoretic differential
    \[
    \diff\varphi_S\colon T_{\Pb^1}\longrightarrow\varphi_S^*T_{\overline{\cM}_{1,1}}
    \]
    This differs from the ramification divisor of the coarse $j$-map at the orbifold values $j=0,1728$. The all-$\I_1$ condition makes $\varphi_S$ non-constant and unramified over the boundary.

    By \cite[Lem.~2.1 and Prop.~2.8(4)]{SB_Torelli}, we have $T_{\overline{\cM}_{1,1}}\simeq\cO(10)$ and $\varphi_S^*\cO(1)\simeq\cO_{\Pb^1}(n)$. Thus $\diff\varphi_S$ is a non-zero section of $\cO_{\Pb^1}(10n-2)$, and
    \[
    \deg\operatorname{Ram}(\varphi_S)=10n-2=N
    \]

    For Weierstrass coefficients $a_i\in H^0(\Pb^1,\cO(in))$, $i=4,6$, this divisor is cut out, including multiplicities, by the weighted Wronskian
    \[
    \Theta(a_4,a_6)\coloneqq3a_6\,\diff a_4-2a_4\,\diff a_6
    \in H^0\!\left(\Pb^1,\cO(10n)\otimes\Omega^1_{\Pb^1}\right)
    \]
    Here the differentials are computed in local frames of $\cO(n)$. Under a change of frame by a unit $u$, the terms involving $\diff u$ cancel, giving $\Theta(u^{-4}a_4,u^{-6}a_6)=u^{-10}\Theta(a_4,a_6)$; hence the expression defines the stated global section. This agrees, up to a non-zero constant, with Shepherd--Barron's expression for the ramification divisor \cite[\S10, Lem.~10.2]{SB_Torelli}. Along $\varphi_S$, the pullback tangent bundle $\varphi_S^*T_{\overline{\cM}_{1,1}}$ is the quotient of $\cO_{\Pb^1}(4n)\oplus\cO_{\Pb^1}(6n)$ by the trivial subbundle generated by $(4a_4,6a_6)$, which is nowhere zero since $a_4$ and $a_6$ have no common zero. Taking the determinant with $(4a_4,6a_6)$ as the first column identifies this quotient with $\cO_{\Pb^1}(10n)$ and represents $\diff\varphi_S$ by $4a_4\,\diff a_6-6a_6\,\diff a_4=-2\Theta(a_4,a_6)$, so the zero divisors agree, including multiplicities.

    \begin{defn}\label{def:simply_branched}
    A surface $[S]\in\cW_n^{\I_1}$ is \emph{simply-branched} if the stack theoretic ramification divisor $\operatorname{Ram}(\varphi_S)$ is reduced. We denote the corresponding substack by $\Wsb\subset\cW_n^{\I_1}$. Thus, after passing to an \'etale neighbourhood which orders the ramification divisor, a point of $\Wsb$ has
    \[
    \operatorname{Ram}(\varphi_S)=p_1+\cdots+p_N
    \]
    with the $p_i$ pairwise distinct. This definition does not require the images $\varphi_S(p_i)$ to be pairwise distinct.
    \end{defn}

    \begin{lem}\label{lem:sb-open-dense}
    The substack $\Wsb$ is Zariski open and dense in $\cW_n^{\I_1}$.
    \end{lem}

    \begin{proof}
    Reducedness of the zero divisor of $\Theta(a_4,a_6)$ is a Zariski open condition on the Weierstrass coefficients. Thus $\Wsb$ is Zariski open in $\cW_n^{\I_1}$. To prove density, we show that multiple ramification points can be split by deforming the surface.

    We use Shepherd--Barron's local deformation theory. The part of his stack $\mathcal{J}\mathcal{E}^{\mathrm{gen}}$ with height $n$ and base of genus zero is our all-$\I_1$ locus \cite[p.~2523]{SB_Torelli}. Its classifying morphisms are non-constant and unramified over the boundary (see the local calculation in the proof of \cite[Prop.~11]{HP}); their ramification divisors need not be reduced. His numerical hypotheses hold because the base has genus zero and $p_g=n-1>0$. Moreover, the relatively minimal surface coincides with its minimal Weierstrass model, since there are no fibre components disjoint from the zero section.

    Here $\operatorname{Def}(S)$ denotes the miniversal analytic deformation space of the Jacobian elliptic fibration with its zero section, with the base considered up to isomorphism. It underlies a local quotient chart of $\cW_n^{\I_1}$ at $[S]$ \cite[Prop.~2.9(2)--(3) and Prop.~4.4(1)]{SB_Torelli}.

    Suppose that the ramification divisor of a given surface has multiplicity $e_i-1$ at a point $p_i$. Shepherd--Barron's local isomorphism \cite[Prop.~4.4]{SB_Torelli} identifies this miniversal chart with the product of the deformation spaces of the finite schemes $\Spec\Cb[t]/(t^{e_i})$. The $i$-th local deformation base has dimension $e_i-1$. Since $\sum_i(e_i-1)=\deg\operatorname{Ram}(\varphi_S)=N$, their product is an $N$ dimensional miniversal chart. On the simply-branched locus, each factor is $1$ dimensional, giving one deformation parameter for each ramification point.

    More generally, the product description allows us to deform the ramification at the different points independently. To split a point of multiplicity $e-1$, lift $\varphi_S$ locally to an \'etale chart of the target stack and choose coordinates $w$ on the source and $s$ on that chart. Shepherd--Barron's Schiffer deformation has the form \cite[Prop.~4.3]{SB_Torelli}
    \[
    s=w^e+e\sum_{j=0}^{e-2}t_{e-j}w^j
    \]
    Up to the non-zero factor $e$, its critical divisor is therefore cut out by
    \[
    w^{e-1}+\sum_{j=1}^{e-2}j\,t_{e-j}w^{j-1}
    \]
    For $e=2$ this divisor is already reduced. For $e\ge3$, setting $t_{e-1}\ne0$ and all other relevant parameters equal to zero gives the polynomial $w^{e-1}+t_{e-1}$, which has distinct roots. Thus the discriminant of the critical polynomial, viewed as a function of the deformation parameters, is not identically zero.

    Choose pairwise disjoint coordinate discs around the points $p_i$. By Rouch\'e's theorem, after shrinking the deformation chart each disc contains exactly $e_i-1$ critical points, counted with multiplicity. These account for the total ramification degree $N$, so there are no critical points outside their union. The simultaneous non-vanishing of the local discriminants defines a dense open subset of the miniversal chart. Thus every surface in $\cW_n^{\I_1}$ is an analytic limit of surfaces with reduced ramification.

    Since $\cW_n^{\I_1}$ is irreducible by the quotient description in \cite[Thms.~1.1 and~1.2(a)]{JJ}, its non-empty Zariski open substack $\Wsb$ is Zariski dense.
    \end{proof}

    \begin{lem}\label{lem:sb-dense-minimal}
    For every $n\ge2$, both inclusions
    \[
    \Wsb\subset\cW_n^{\I_1}\subset\cW_{\min,n}
    \]
    are Zariski open and dense. In particular, $\Wsb$ is analytically dense in $\cW_{\min,n}$.
    \end{lem}

    \begin{proof}
    Minimality is a Zariski open condition on the Weierstrass data in the irreducible coefficient space $V_n$ \cite[Prop.~4.15]{BPS}. Intersecting with the open condition $\Delta\not\equiv0$ gives a non-empty Zariski open subset of $V_n$ presenting $\cW_{\min,n}$ after quotienting by Weierstrass scaling and changes of coordinates on the base \cite[Thms.~1.1--1.2]{JJ}.

    Reducedness of the discriminant divisor on $\Pb^1$ is also a Zariski open condition. This condition is non-empty: in homogeneous coordinates $[u:v]$ on $\Pb^1$, take $a_4=-3u^{4n}$ and $a_6=v^{6n}$. Then
    \[
    \Delta=-16(4a_4^3+27a_6^2)=-432(v^{12n}-4u^{12n})
    \]
    has $12n$ distinct zeros, and it does not vanish at $[u:v]=[0:1]$. These Weierstrass data are minimal. Indeed, non-minimality at a point would force the discriminant to vanish there to order at least $12$.

    Thus the squarefree discriminant locus is a non-empty Zariski open subset of the irreducible minimal elliptic coefficient locus. Passing to the quotient shows that $\cW_n^{\I_1}$ is Zariski open and dense in $\cW_{\min,n}$.

    Lemma~\ref{lem:sb-open-dense} gives the same conclusion for $\Wsb\subset\cW_n^{\I_1}$. Hence $\Wsb$ is Zariski open and dense in $\cW_{\min,n}$. A non-empty Zariski open subset of an irreducible complex variety is analytically dense; applying this on the coefficient presentation gives the analytic assertion.
    \end{proof}

    \subsection{The Shepherd--Barron frame}

    We now choose coordinates and a cohomology basis in which the first derivative of the period map is explicit. This will let us test transversality of the period equations and show that bounded height Hodge classes can be avoided locally.

    Fix $[S]\in\Wsb$ and order its ramification points as $p_1,\ldots,p_N$. Choose an affine coordinate $z$ on the base for which none of the $p_i$ lies at infinity, and put $z_i=z(p_i)$. On the corresponding affine Weierstrass chart, write
    \[
    y^2=x^3+a_4(z)x+a_6(z)
    \qquad
    \omega_0=\frac{\diff x\wedge\diff z}{2y}
    \]
    The canonical bundle formula identifies global holomorphic two-forms with polynomials in $z$ of degree at most $n-2$, multiplied by $\omega_0$. Thus
    \[
    H^{2,0}(S)=\bigoplus_{j=0}^{h-1}\Cb\,\omega_j
    \qquad
    \omega_j=z^j\omega_0
    \]
    The form $\omega_0$ is nowhere vanishing over the affine $z$-line. We use this polynomial identification for sections of $\cO_{\Pb^1}(n-2)$ below; in particular, $g(p_i)$ means the scalar $g(z_i)$.

    Every automorphism of the Jacobian fibration preserves its zero section and the reduced discriminant divisor of degree $12n$. Its induced automorphism of $\Pb^1$ therefore belongs to the finite group preserving these $12n\ge3$ points, while the kernel embeds in the finite group $\operatorname{Aut}_{\Cb(z)}(E_S,0)$. Thus $G\coloneqq\operatorname{Aut}_{\Wsb}([S])$ is finite. Since $\Wsb$ is an open substack of $\cW_{\min,n}$, it is smooth and separated Deligne--Mumford by \cite[Thm.~1.2(a)--(b)]{JJ}. By \cite[Prop.~2.9]{Eyssidieux_Sala}, it has an analytic open neighbourhood of the form $[Z/K]$ with $K$ finite. Choose a lift of $[S]$ to $Z$ and identify its stabilizer in $K$ with $G$. Choose a $G$-invariant neighbourhood of this lift disjoint from its translates under $K\setminus G$; its quotient by $G$ is then an open substack of $[Z/K]$. Linearizing the $G$-action near the fixed lift gives an analytic open immersion
    \[
    [B/G]\longrightarrow\Wsb
    \]
    where $B$ is a simply-connected analytic ball and $0\in B$ lifts $[S]$ and is fixed by $G$. The quotient $B/G$ identifies with an analytic neighbourhood of the image of $[S]$ in the coarse moduli space; in particular, two points of $B$ represent isomorphic Jacobian elliptic fibrations precisely when they lie in the same $G$-orbit.

    The relative ramification divisor is finite \'etale of degree $N$ over $B$. Since $B$ is simply-connected, the ordering at $0$ extends to an ordering throughout $B$. We trivialize the primitive integral local system over $B$ by flat transport. By the preceding identification of deformation spaces and \cite[\S4, Prop.~4.4]{SB_Torelli}, after shrinking $B$ there are analytic coordinates $(t_1,\ldots,t_N)$ such that $\left.\partial/\partial t_i\right|_0$ spans the Schiffer direction associated with $p_i$.

    Shepherd--Barron associates to each $p_i$ a meromorphic two-form $\widetilde\eta_i$ with poles of order at most two along the fibre over $p_i$ and no other poles. It is of the second kind, meaning that its residue along that fibre vanishes. Its class is taken in the primitive quotient $F^\perp/\Cb F$. The zero section identifies this quotient with $\Lambda_{n,\Cb}$ by
    \[
    [x]\longmapsto x-(x\cdot O)F
    \]
    This expression is independent of the representative $x\in F^\perp$ and is orthogonal to both $F$ and $O$. Let $\eta_i=\operatorname{pr}^{1,1}[\widetilde\eta_i]\in H^{1,1}_{\mathrm{prim}}(S)$ denote its $(1,1)$-component.

    Shepherd--Barron's meromorphic classes $[\widetilde\eta_i]$ lie in $H^{2,0}(S)\oplus H^{1,1}_{\mathrm{prim}}(S)$ \cite[Thm.~4.10]{SB_Torelli}. The $H^{2,0}(S)$-summand pairs trivially with this space, so $[\widetilde\eta_i]\cdot[\widetilde\eta_j]=\eta_i\cdot\eta_j$. Since the ramification divisor is reduced, the projected classes $\eta_1,\ldots,\eta_N$ form a basis of $H^{1,1}_{\mathrm{prim}}(S)$ orthogonal for the complex bilinear intersection pairing \cite[Thm.~4.11(4)]{SB_Torelli}.

    Write $E_{p_i}=f^{-1}(p_i)$ for the fibre over $p_i$. Let $F^2$ denote the holomorphic Hodge bundle over $B$, whose fibre at $b$ is $H^{2,0}(S_b)$. The Gauss--Manin connection $\nabla$ differentiates sections of the cohomology bundle using its flat trivialization. For every $\omega\in H^{2,0}(S)$, Shepherd--Barron's rank-one formula \cite[Thm.~1.4(3)]{SB_Torelli} gives
    \[
    \left(\nabla_{\partial/\partial t_i}\omega\right)^{1,1}=c_i(\omega)\eta_i
    \]
    The derivative is evaluated at $0\in B$ using any local holomorphic extension of $\omega$ in $F^2$. Its $(1,1)$-projection is independent of that extension: changing the extension only adds an element of $H^{2,0}(S)$ to the derivative; see \cite[Lem.~5.3.2]{CMSP}. Here $c_i$ is a linear functional with
    \[
    \ker(c_i)=H^0(S,\Omega_S^2(-E_{p_i}))
    \]
    Since $K_S\simeq f^*\cO_{\Pb^1}(n-2)$, the canonical linear system is basepoint free for $n\ge2$. Under the polynomial identification above, vanishing along $E_{p_i}$ is the single condition $g(p_i)=0$. Hence this kernel is a hyperplane, and $c_i$ is non-zero.

    Griffiths transversality places the derivative in $H^{2,0}(S)\oplus H^{1,1}_{\mathrm{prim}}(S)$. The kernel description of $c_i$ gives the following evaluation formula.

    \begin{lem}\label{lem:ci-evaluation}
    For every polynomial $g(z)$ of degree at most $n-2$, representing a section of $\cO_{\Pb^1}(n-2)$ in the chosen affine coordinate, one has
    \[
    c_i(g\omega_0)=c_i(\omega_0)g(p_i)
    \qquad
    c_i(\omega_0)\ne0
    \]
    In particular, $c_i(\omega_j)=c_i(\omega_0)z_i^j$ for $0\le j\le h-1$.
    \end{lem}

    \begin{proof}
    Under the polynomial identification above, write $g=g(p_i)+(g-g(p_i))$. Since $g-g(p_i)$ vanishes at $p_i$, the form $(g-g(p_i))\omega_0$ belongs to $\ker(c_i)$. Linearity therefore gives $c_i(g\omega_0)=g(p_i)c_i(\omega_0)$. Since $\omega_0$ does not vanish along $E_{p_i}$, it does not belong to $\ker(c_i)$, so $c_i(\omega_0)\ne0$. Taking $g(z)=z^j$ gives the final formula.
    \end{proof}

    Shepherd-Barron's formula and Lemma~\ref{lem:ci-evaluation} realize the equality $\dim B=h^{1,1}_{\mathrm{prim}}(S)=N$ by a complex linear isomorphism
    \[
    T_0B\xrightarrow{\sim}H^{1,1}_{\mathrm{prim}}(S),
    \qquad
    v\longmapsto(\nabla_v\omega_0)^{1,1}
    \]
    Indeed, this map sends $\left.\partial/\partial t_i\right|_0$ to $c_i(\omega_0)\eta_i$. These images form a basis because the $\eta_i$ form a basis and every $c_i(\omega_0)$ is non-zero.

    \medskip

    The evaluation formula lets us test whether the $h$ period equations for a fixed class have independent differentials. The following criterion does this by selecting $h$ ramification points.

    \begin{thm}\label{thm:KS-transversality}
    Let $\gamma$ be a locally flat real class which is of type $(1,1)$ at $[S]\in\Wsb$. If there are distinct indices $i_1,\ldots,i_h$ such that
    \[
    \gamma\cdot\eta_{i_k}\ne0
    \qquad (1\le k\le h)
    \]
    then, on the marked local deformation space, the fixed class type-$(1,1)$ locus defined by $\gamma$ is smooth of codimension $h=n-1$ at $[S]$.
    \end{thm}

    \begin{proof}
    Extend $\omega_0,\ldots,\omega_{h-1}$ to a holomorphic frame of $F^2$ near $[S]$ and set $f_j=\gamma\cdot\omega_j$ for $0\le j<h$. For a real flat class, vanishing of these pairings is equivalent to being of type $(1,1)$; the conjugate equations follow automatically. Flatness of $\gamma$ and compatibility of the Gauss--Manin connection with the intersection pairing give $\partial f_j/\partial t_i=\gamma\cdot\nabla_{\partial/\partial t_i}\omega_j$. At $[S]$, the $H^{2,0}(S)$-component of the derivative pairs trivially with $\gamma$, since $\gamma$ has type $(1,1)$. Shepherd--Barron's formula and Lemma~\ref{lem:ci-evaluation} therefore give
    \[
    \left.\frac{\partial f_j}{\partial t_{i_k}}\right|_{[S]}
    =
    c_{i_k}(\omega_0)z_{i_k}^j(\gamma\cdot\eta_{i_k})
    \]
    If $J_\gamma$ denotes this $h\times h$ minor, then
    \[
    \det J_\gamma=
    \left(\prod_{k=1}^{h}c_{i_k}(\omega_0)(\gamma\cdot\eta_{i_k})\right)
    \prod_{1\le a<b\le h}(z_{i_b}-z_{i_a})
    \]
    The first product is non-zero by Lemma~\ref{lem:ci-evaluation} and the hypothesis on $\gamma$. The second is non-zero because the ramification points are distinct. Hence the Jacobian of $(f_0,\ldots,f_{h-1})$ has rank $h$ at $[S]$. These equations define the fixed class type-$(1,1)$ locus, so the holomorphic implicit function theorem makes this locus smooth of codimension $h$ at $[S]$.
    \end{proof}

    For $n\ge3$, the following statement is a consequence of Cox's result that a very general surface has Picard number two \cite[Prop.~3.2]{Cox_NL}; the case $n=2$ is classical. We include a proof using Schiffer deformation coordinates.

    \begin{prop}\label{prop:very-general-rank-zero}
    For every $n\ge2$, a very general member of $\cW_{\min,n}$ has Mordell--Weil rank zero.
    \end{prop}

    \begin{proof}
    At any point of $\Wsb$, every non-zero integral Hodge class in $\Lambda_n$ pairs non-trivially with some Schiffer image class. The rank-one formula therefore gives a deformation direction in which one of its periods has non-zero derivative. Consequently, every non-zero fixed class Hodge locus on a marked local chart is a proper closed analytic subset. Since the lattice is countable, Baire's theorem shows that their union has empty analytic interior.

    Apply the algebraicity theorem of Cattani--Deligne--Kaplan \cite[Thm.~1.1 and Cor.~1.3]{CDK} to the Tate twist of the primitive variation on a smooth algebraic atlas of finite type for $\Wsb$, so that primitive $(1,1)$-classes become classes of type $(0,0)$. The locus carrying a non-zero integral Hodge class is a countable union of closed algebraic subsets of this atlas. Their images in $\Wsb$ are constructible. None is dense, since a dense constructible image would contain a non-empty open subset, contradicting the local assertion about empty interior.

    Their closures in $\cW_{\min,n}$, together with the proper closed complement of $\Wsb$, form a countable union of proper closed substacks. Outside this union there are no non-zero integral Hodge classes in $\Lambda_n$. Since all singular fibres are irreducible, the Shioda--Tate formula gives Mordell--Weil rank zero.
    \end{proof}

    For the height bounds, we need a stronger local avoidance statement: on a relatively compact marked chart, only finitely many non-zero integral classes with bounded $-\gamma^2/2$ can become Hodge.

    \begin{lem}\label{lem:bounded-height-avoidance}
    Let $B_0$ be a marked local deformation chart in $\Wsb$, and let $B\subset B_0$ be a connected open subset whose closure is compact and contained in $B_0$. For every $H>0$, the locus
    \[
    B_{\le H}
    =
    \left\{
    b\in B:
    \text{there exists }0\ne\gamma\in\Lambda_n\cap H^{1,1}(S_b)
    \text{ with }-\frac12\gamma^2\le H
    \right\}
    \]
    is a proper closed analytic subset of $B$. Its complement $B_{>H}$ is therefore open and dense.
    \end{lem}

    \begin{proof}
    Fix a Euclidean norm on $\Lambda_{n,\Rb}$. The primitive real $(1,1)$-spaces vary continuously over $\overline B$, and the intersection pairing is negative definite on each of them. Compactness therefore gives a constant $c>0$ such that
    \[
    -\gamma^2\ge c\|\gamma\|^2
    \qquad
    \left(
    b\in\overline B,\ 
    \gamma\in H^{1,1}_{\mathrm{prim}}(S_b,\Rb)
    \right)
    \]
    Consequently, every non-zero integral Hodge class with $-\gamma^2/2\le H$ belongs to the finite set
    \[
    A_H
    =
    \left\{
    \gamma\in\Lambda_n:
    0<-\frac12\gamma^2\le H,\ 
    \|\gamma\|^2\le\frac{2H}{c}
    \right\}
    \]
    Thus $B_{\le H}$ is the union of the fixed class Hodge loci of the classes in $A_H$. Each is closed analytic and proper by the period calculation in the proof of Proposition~\ref{prop:very-general-rank-zero}. Their finite union is proper because $B$ is connected.
    \end{proof}

    \subsection{Simultaneous period equations}

    For $r$ classes, we need $rh$ period equations with independent differentials. We first construct a complex tuple for which the Jacobian splits into $r$ Vandermonde blocks, then use polynomial non-vanishing to obtain a real tuple.

    \begin{lem}\label{lem:simultaneous-SB-minor}
    Let $r\ge1$, put $M=rh$, and suppose $M\le N$. Put $W\coloneqq H^{1,1}_{\mathrm{prim}}(S,\Rb)$. There exist $\Rb$-linearly independent vectors $v_1,\ldots,v_r\in W$ with $v_a^2=-1$ for $1\le a\le r$, and a subset $I\subset\{1,\ldots,N\}$ of cardinality $M$, such that the $M\times M$ matrix
    \[
    J_{(a,j),i}
    =
    c_i(\omega_0)z_i^j(v_a\cdot\eta_i)
    \qquad
    1\le a\le r,\quad 0\le j\le h-1,\quad i\in I
    \]
    is non-singular.
    \end{lem}

    \begin{proof}
    Choose any $M$ indices and partition them into disjoint blocks $I=C_1\sqcup\cdots\sqcup C_r$ with $|C_a|=h$. Over $\Cb$, set
    \[
    v_a^\circ=\sum_{i\in C_a}\frac{\eta_i}{\eta_i^2}
    \]
    The denominators are non-zero because the $\eta_i$ form an orthogonal basis of a non-degenerate quadratic space. Orthogonality gives
    \[
    v_a^\circ\cdot\eta_i
    =
    \begin{cases}
    1,&i\in C_a\\
    0,&i\notin C_a
    \end{cases}
    \]
    After ordering rows and columns by the blocks, the resulting matrix is block diagonal. Its $a$-th block has determinant
    \[
    \left(\prod_{i\in C_a}c_i(\omega_0)\right)
    \prod_{\substack{i,k\in C_a\\i<k}}(z_k-z_i)
    \]
    which is non-zero.

    Thus the determinant is a non-zero polynomial on $(W\otimes_{\Rb}\Cb)^r$. A complex polynomial which vanishes on the real form $W^r$ vanishes identically. Hence we may choose a real tuple on which the determinant is non-zero. The chosen vectors are linearly independent: any relation $\sum_a\lambda_av_a=0$ would give the same relation among the rows indexed by $(a,0)$, contradicting determinant non-vanishing.

    Finally, negative definiteness allows us to replace each $v_a$ by $v_a/\sqrt{-v_a^2}$. For each $a$, the $h$ rows indexed by $(a,j)$, $0\le j<h$, are multiplied by the non-zero scalar $1/\sqrt{-v_a^2}$, so the determinant remains non-zero.
    \end{proof}

    To pass from the real tuple to integral classes, we approximate its directions by indivisible lattice vectors of increasing norm. These approximations may leave the original Hodge subspace; we will subsequently deform the surface to make the chosen integral classes Hodge.

    \begin{lem}\label{lem:primitive-directions}
    Let $L\subset E$ be a full lattice, meaning a discrete subgroup spanning the finite dimensional real vector space $E$. Assume $\dim E\ge2$ and equip $E$ with a Euclidean norm. Let $v_1,\ldots,v_r\in E$ be linearly independent unit vectors. There exist indivisible vectors $\gamma_{a,m}\in L$, for $1\le a\le r$, such that
    \[
    \|\gamma_{a,m}\|\longrightarrow\infty
    \qquad
    \frac{\gamma_{a,m}}{\|\gamma_{a,m}\|}\longrightarrow v_a
    \]
    and, for all sufficiently large $m$, the vectors $\gamma_{1,m},\ldots,\gamma_{r,m}$ are linearly independent. If a proper real subspace $E_0\subsetneq E$ contains all the $v_a$, the vectors $\gamma_{a,m}$ may moreover be chosen outside $E_0$.
    \end{lem}

    \begin{proof}
    After choosing a basis of $L$, the directions of indivisible lattice vectors are precisely the rational directions, which are dense in the unit sphere of $E$. For each $m$, only finitely many indivisible vectors have norm at most $m$; removing their directions leaves a dense subset of the unit sphere.

    If $E_0\subsetneq E$ is a proper real subspace, rational directions outside $E_0$ are still dense: every open subset of the unit sphere contains a smaller open subset disjoint from $E_0$, and rational directions are dense there. For each $a$ and $m$, we may therefore choose an indivisible vector $\gamma_{a,m}$, outside $E_0$ when specified, such that
    \[
    \|\gamma_{a,m}\|>m
    \qquad
    \left\|\frac{\gamma_{a,m}}{\|\gamma_{a,m}\|}-v_a\right\|<\frac1m
    \]
    These inequalities give the two asserted limits. Since linear independence is an open condition on tuples of vectors, convergence to the independent tuple $(v_1,\ldots,v_r)$ gives independence for all sufficiently large $m$.
    \end{proof}

    \subsection{Local simultaneous realization}

    The implicit function theorem lets us compensate for a small change in the cohomology classes by a small deformation of the surface. We apply this to indivisible lattice vectors whose normalized directions approach the real tuple constructed above, while their norms tend to infinity.

    \begin{thm}\label{thm:local-simultaneous-realization}
    Let $n\ge2$, let $[S]\in\Wsb$, and let $r\ge1$ satisfy $r(n-1)\le10n-2$. Then there are points $b_m\to0$ in the marked local deformation chart $B$ chosen above, representing pairwise non-isomorphic Jacobian elliptic surfaces,
    \[
    [S_m]\in\Wsb
    \qquad
    [S_m]\longrightarrow[S]
    \]
    and, after identifying the primitive cohomology lattices by flat transport, indivisible classes $\gamma_{1,m},\ldots,\gamma_{r,m}\in\Lambda_n$ with the following properties:
    \begin{enumerate}[\normalfont(1)]
    \item every $\gamma_{a,m}$ is of type $(1,1)$ on $S_m$;
    \item the classes $\gamma_{1,m},\ldots,\gamma_{r,m}$ span a negative-definite subspace and are linearly independent;
    \item on the marked local deformation space, the $r(n-1)$ period equations defining their simultaneous fixed class Hodge locus have linearly independent differentials at $[S_m]$; in particular, this locus is smooth of complex codimension $r(n-1)$ there;
    \item for every $a$,
    \[
    d_{a,m}\coloneqq-\frac{\gamma_{a,m}^2}{2}\in\bZ_{>0}
    \qquad
    d_{a,m}\longrightarrow\infty
    \]
    \end{enumerate}
    \end{thm}

    \begin{proof}
    Put $M=rh$. Work on the simply-connected chart $B$ above. Flat transport identifies the real primitive cohomology of its fibres with $\Lambda_{n,\Rb}$. At the central surface, Hodge decomposition gives an orthogonal splitting
    \[
    \Lambda_{n,\Rb}=H_+\oplus W
    \]
    where $H_+=\bigl(H^{2,0}(S)\oplus H^{0,2}(S)\bigr)\cap\Lambda_{n,\Rb}$ and $W=H^{1,1}_{\mathrm{prim}}(S,\Rb)$. The intersection form is positive definite on $H_+$ and negative definite on $W$. We therefore use the positive-definite norm determined by
    \[
    \|v+w\|^2=v^2-w^2
    \qquad
    v\in H_+,\quad w\in W
    \]
    In particular, the vectors supplied by Lemma~\ref{lem:simultaneous-SB-minor} have norm one.

    Choose $v_1,\ldots,v_r\in W$ and $I\subset\{1,\ldots,N\}$ as in Lemma~\ref{lem:simultaneous-SB-minor}. Let $s=(t_i)_{i\in I}$ be the corresponding $M$ coordinates, and let $B_I\subset B$ be the slice on which all other coordinates are zero. Restrict $B_I$ to a sufficiently small coordinate polydisc $\mathbb D^M\subset B$ with coordinates $s$, where $\mathbb D\subset\Cb$ is an open disc centred at zero. Choose a holomorphic frame $\omega_0(t),\ldots,\omega_{h-1}(t)$ of $F^2$ over $B$ whose value at the origin is the monomial basis above. With the notation $\omega_j(s)=\omega_j(t)|_{B_I}$, define the restriction of the Hodge-equation map to this slice by
    \[
    \begin{gathered}
    \Psi\colon(\Lambda_{n,\Cb})^r\times\mathbb D^M\longrightarrow\Cb^{rh}\\
    \Psi_{a,j}(\boldsymbol u,s)=u_a\cdot\omega_j(s)
    \qquad
    1\le a\le r,\quad 0\le j\le h-1
    \end{gathered}
    \]
    Because $v_a\in W$, one has $\Psi(\boldsymbol v,0)=0$. Moreover, the same calculation as in the proof of Theorem~\ref{thm:KS-transversality}, using Lemma~\ref{lem:ci-evaluation}, gives
    \[
    \left.\frac{\partial\Psi_{a,j}}{\partial t_i}\right|_{(\boldsymbol v,0)}
    =
    c_i(\omega_0)z_i^j(v_a\cdot\eta_i)
    \]
    so the derivative with respect to $s$ is non-singular by Lemma~\ref{lem:simultaneous-SB-minor}.

    The holomorphic implicit function theorem therefore gives a neighbourhood $\mathcal V$ of $\boldsymbol v=(v_1,\ldots,v_r)$ in $(\Lambda_{n,\Cb})^r$ and a holomorphic map $\tau\colon\mathcal V\to\mathbb D^M$ with $\tau(\boldsymbol v)=0$ such that
    \begin{equation}\label{eq:IFT-exact-zero}
    \Psi(\boldsymbol u,\tau(\boldsymbol u))=0
    \end{equation}
    for every $\boldsymbol u\in\mathcal V$. Shrinking $\mathcal V$ if necessary, the derivative with respect to $s$ remains non-singular along the graph of $\tau$.

    The period equations are homogeneous in each class, so normalization does not change their zero locus. Since $h>0$, the positive subspace $H_+$ is non-zero and $W\subset\Lambda_{n,\Rb}$ is proper. We may therefore apply Lemma~\ref{lem:primitive-directions} to the lattice $\Lambda_n\subset\Lambda_{n,\Rb}$ with $E_0=W$. Choose indivisible vectors $\gamma_{a,m}\notin W$ such that
    \[
    u_{a,m}\coloneqq\frac{\gamma_{a,m}}{\|\gamma_{a,m}\|}\longrightarrow v_a
    \qquad
    \|\gamma_{a,m}\|\longrightarrow\infty
    \]
    For large $m$, the tuple $\boldsymbol u_m$ lies in $\mathcal V$. Let $b_m\in B$ have selected coordinates $\tau(\boldsymbol u_m)$ and all remaining coordinates zero, and let $S_m$ be the corresponding surface. Then $b_m\to0$, and \eqref{eq:IFT-exact-zero} yields
    \[
    \gamma_{a,m}\cdot\omega=0
    \qquad
    \text{for every }\omega\in F^2(S_m)
    \quad\text{and every }1\le a\le r
    \]
    The classes $\gamma_{a,m}$ are real, so they are also orthogonal to $\overline{F^2(S_m)}$ and hence are of type $(1,1)$.

    For the fixed tuple $(\gamma_{1,m},\ldots,\gamma_{r,m})$, the simultaneous Hodge locus in the full chart $B$ is defined by the $rh$ equations
    \[
    \gamma_{a,m}\cdot\omega_j(t)=0
    \qquad
    1\le a\le r,\quad 0\le j\le h-1
    \]
    Their Jacobian contains the selected $M\times M$ minor from the implicit function argument. Replacing $u_{a,m}$ by $\gamma_{a,m}$ multiplies each of the $h$ rows belonging to $a$ by $\|\gamma_{a,m}\|$, so this minor remains non-singular. The full Jacobian therefore has rank $M=rh$, and the simultaneous fixed class Hodge locus is smooth of codimension $M$ at $[S_m]$.

    Each $\gamma_{a,m}$ is a non-zero class in $H^{1,1}_{\mathrm{prim}}(S_m,\Rb)$, where the intersection pairing is negative definite. Thus their span is negative definite, and each $d_{a,m}$ is positive. By Lemma~\ref{lem:primitive-directions}, the classes are linearly independent for all sufficiently large $m$. Each $d_{a,m}$ is an integer because $\Lambda_n$ is even. Since $v_a^2=-1$, one has
    \[
    \frac{d_{a,m}}{\|\gamma_{a,m}\|^2}
    =
    -\frac{\gamma_{a,m}^2}{2\|\gamma_{a,m}\|^2}
    \longrightarrow\frac12
    \]
    Therefore $d_{a,m}\to\infty$.

    Finally, $\gamma_{a,m}\notin W$ implies that it is not of type $(1,1)$ at the central surface, so $\tau(\boldsymbol u_m)\ne0$. These non-zero points tend to zero; after passing to a subsequence they are pairwise distinct. Since $G$ is finite, a further subsequence lies in pairwise distinct $G$-orbits. By the choice of the uniformizing chart, their images in the coarse analytic moduli space are pairwise distinct; equivalently, the corresponding Jacobian elliptic surfaces are pairwise non-isomorphic. All the resulting points lie in $\Wsb$ by construction. After discarding finitely many initial terms and relabelling the final subsequence, properties~\textup{(1)}--\textup{(4)} hold for every $m$.
    \end{proof}

    \subsection{Mordell--Weil rank, section heights, and analytic density}

    Since every fibre is irreducible, the Shioda--Tate description removes only the fibre and zero section classes. The primitive integral Hodge classes constructed above therefore give Mordell--Weil sections, and their intersection pairing determines the canonical heights. By avoiding bounded height Hodge classes before applying local realization, we make the least non-zero section height tend to infinity. We then use density of the simply-branched locus to obtain statements throughout the minimal elliptic locus.

    \begin{thm}\label{thm:KS-higher-rank}
    Let $n,r$ be integers satisfying $n\ge2$ and
    \[
    1\le r\le\left\lfloor\frac{10n-2}{n-1}\right\rfloor
    \]
    For every $[S]\in\Wsb$ there are pairwise non-isomorphic Jacobian elliptic surfaces
    \[
    [S_m]\in\Wsb
    \qquad
    [S_m]\longrightarrow[S]
    \]
    and linearly independent sections $P_{1,m},\ldots,P_{r,m}\in\MW(S_m/\Pb^1)$ such that
    \[
    \min_{0\ne P\in\MW(S_m/\Pb^1)}\widehat h(P)\longrightarrow\infty
    \]
    In particular, $\rk\MW(S_m/\Pb^1)\ge r$. The corresponding Shioda classes are indivisible in $\Lambda_n$, and the canonical heights of non-zero sections are positive integers. On the marked local deformation chart, the $r(n-1)$ period equations for these classes have linearly independent differentials at $[S_m]$. Hence their simultaneous fixed class Hodge locus is smooth of complex codimension $r(n-1)$ there.
    \end{thm}

    \begin{proof}
    Choose a marked local deformation chart $B$ centred at a point $b_0$ representing $[S]$, relatively compact in a larger marked chart, and a decreasing neighbourhood basis $U_m$ of $b_0$. By Lemma~\ref{lem:bounded-height-avoidance}, each $U_m\cap B_{>m}$ is non-empty and open. Choose a point in this intersection and apply Theorem~\ref{thm:local-simultaneous-realization} on a sufficiently small marked deformation chart contained in $U_m\cap B_{>m}$. Its sequence contains infinitely many pairwise non-isomorphic surfaces, and all sufficiently late terms have representatives in $U_m\cap B_{>m}$.

    Choose one such representative $b_m$, with $S_m=S_{b_m}$ not isomorphic to any previously chosen surface, together with its $r$ independent indivisible integral Hodge classes $\gamma_{1,m},\ldots,\gamma_{r,m}$. Use the marking on $B$ to regard these classes as elements of $\Lambda_n$. Then $b_m\to b_0$, and the simultaneous period differentials are independent at $b_m$. Put $d_{a,m}=-\gamma_{a,m}^2/2$. Since $b_m\in B_{>m}$, we have $d_{a,m}>m$ for every $a$.

    Since $S_m$ has singular fibres, its N\'eron--Severi group is torsion-free \cite[Thm.~1.2]{Shioda}. We therefore identify $\NS(S_m)$ with its image in $H^2(S_m,\bZ)/\mathrm{tors}$. The Lefschetz $(1,1)$-theorem gives $\gamma_{a,m}\in\NS(S_m)$ for every $a$ \cite[\S3.4, pp.~119--120]{CMSP}.

    All fibres of $S_m$ are irreducible, so $\Triv(S_m)=\langle F,O\rangle$. Since this lattice is unimodular and $F$ and $O$ are algebraic, the integral orthogonal splitting from Section~\ref{sec:primitive-cohomology} restricts to
    \[
    \NS(S_m)=\Triv(S_m)\oplus\bigl(\NS(S_m)\cap\Lambda_n\bigr)
    \]
    The Shioda--Tate description \cite[Thm.~1.3]{Shioda} therefore gives
    \[
    \MW(S_m/\Pb^1)
    \simeq
    \NS(S_m)/\Triv(S_m)
    \simeq
    \NS(S_m)\cap\Lambda_n
    \]
    In particular, $\MW(S_m/\Pb^1)$ is torsion-free. Under this identification, the map from sections to algebraic classes in $\Lambda_n$ is Shioda's orthogonal projection $\sigma_{\mathrm{Sh}}$ \cite[Lem.~8.1]{Shioda}. Consequently, each $\gamma_{a,m}$ determines a unique section $P_{a,m}$ with $\sigma_{\mathrm{Sh}}(P_{a,m})=\gamma_{a,m}$. Since this correspondence is an isomorphism of abelian groups, the linear independence of the classes gives the linear independence of the sections.

    Shioda's height pairing \cite[Thm.~8.4]{Shioda}, with our half-normalization, gives
    \[
    \widehat h(P)=-\frac12\,\sigma_{\mathrm{Sh}}(P)^2
    \]
    Every non-zero section $P\in\MW(S_m/\Pb^1)$ has a non-zero integral Shioda class of type $(1,1)$. Since $b_m\in B_{>m}$, every such section satisfies $\widehat h(P)>m$. The lattice $\Lambda_n$ is even and its real primitive $(1,1)$-subspace is negative definite, so these heights are positive integers. The Mordell--Weil group has positive rank, hence the set of non-zero section heights is non-empty and has a minimum. Therefore
    \[
    \min_{0\ne P\in\MW(S_m/\Pb^1)}\widehat h(P)>m
    \]
    which proves the asserted divergence.

    The holomorphic implicit function theorem gives the stated smoothness and codimension of the simultaneous fixed class Hodge locus.
    \end{proof}

    \begin{cor}\label{cor:rank-density}
    Let $n\ge2$ and $r\ge1$ satisfy $r(n-1)\le10n-2$. The locus
    \[
    \left\{
    [S]\in\Wsb:
    \rk\MW(S/\Pb^1)\ge r
    \right\}
    \]
    is analytically dense in $\cW_{\min,n}$.

    More precisely, for every non-empty analytic open subset $U\subset\cW_{\min,n}$ and every $H>0$, there are infinitely many pairwise non-isomorphic surfaces $[S']\in U\cap\Wsb$ carrying $r$ independent sections with indivisible Shioda classes and satisfying $\widehat h(P)>H$ for every $0\ne P\in\MW(S'/\Pb^1)$.

    The $r(n-1)$ simultaneous fixed class period equations are transverse at $[S']$. On a marked local deformation chart $B\to\Wsb$ at $[S']$, their Hodge germ has a smooth representative $Z\subset B$ of dimension $10n-2-r(n-1)$, with image in $U\cap\Wsb$, such that
    \[
    \widehat h(P)>H
    \qquad
    \text{for every }b\in Z
    \text{ and every }0\ne P\in\MW(S_b/\Pb^1)
    \]
    \end{cor}

    \begin{proof}
    By Lemma~\ref{lem:sb-dense-minimal}, choose $[S]\in U\cap\Wsb$. Theorem~\ref{thm:KS-higher-rank} gives pairwise non-isomorphic surfaces $[S_m]\to[S]$ carrying the required independent sections and with least non-zero section height tending to infinity. All sufficiently late terms lie in $U$ and satisfy the prescribed bound. Indivisibility and transversality follow from the same theorem.

    At each selected surface, the integral Shioda identification shows that there is no non-zero integral Hodge class $\gamma$ in primitive cohomology with $-\gamma^2/2\le H$. Lemma~\ref{lem:bounded-height-avoidance} therefore gives an open neighbourhood with the same property. Shrinking the simultaneous Hodge germ into this neighbourhood and into the inverse image of $U$ gives the required representative.
    \end{proof}

    For $r=1$, every non-empty analytic open subset of $\cW_{\min,n}$ contains a sequence of pairwise non-isomorphic surfaces in $\Wsb$ carrying sections with indivisible Shioda classes and strictly increasing canonical heights. The density and arbitrarily large lower bounds for all non-zero section heights recover Cox's theorem \cite[Thm.~4.4]{Cox_NL}.

    The proof above supplies explicit transverse period minors and smooth Hodge locus representatives on which all non-zero section heights satisfy a prescribed uniform lower bound. Subsection~\ref{subsec:two-pointed-moduli} identifies the corresponding section moduli germs.

    \subsection{Exact Mordell--Weil ranks and first-order criteria}
    \label{subsec:er1}

    We prove exact-rank density together with uniform lower bounds for the heights of all non-zero sections. We state the conclusion first; the proof is given below, with the rank $10$ case treated in the next subsection.

    \begin{thm}\label{thm:er-density}
    Let $n\ge2$ and $r\ge1$ be integers, with $r\le18$ if $n=2$ and $r\le10$ if $n\ge3$. Then the locus
    \[
    \left\{
    [S]\in\Wsb:
    \rk\MW(S/\Pb^1)=r
    \right\}
    \]
    is analytically dense in $\cW_{\min,n}$.

    More precisely, for every non-empty analytic open subset $U\subset\cW_{\min,n}$ and every $H>0$, there are a marked deformation chart $B\to\Wsb$ parametrizing surfaces $S_b$ and a smooth connected local simultaneous fixed class Hodge locus $Z\subset B$ of dimension $10n-2-r(n-1)$, defined by $r$ independent indivisible integral classes, with image in $U\cap\Wsb$, such that the Mordell--Weil rank is exactly $r$ outside a countable union of proper closed analytic subsets of $Z$ and
    \[
    \widehat h(P)>H
    \qquad
    \text{for every }b\in Z
    \text{ and every }0\ne P\in\MW(S_b/\Pb^1).
    \]
    \end{thm}

    The imposed classes already give a lower bound on Mordell--Weil rank. To obtain exact rank, we must prevent any integral class outside their rational span from remaining Hodge throughout the constructed family. On a connected smooth simultaneous Hodge locus, each additional integral class either remains Hodge everywhere or has a proper closed analytic Hodge locus. If the first possibility is excluded for every such class, countability of the lattice and Baire's theorem give exact rank outside a countable union of proper closed analytic subsets.

    We will exclude persistent additional classes by differentiating their period equations along the imposed Hodge locus. A non-zero restricted derivative shows that the corresponding class cannot remain Hodge identically. This first-order method suffices for $1\le r\le9$ when $n\ge3$ and for $1\le r\le18$ when $n=2$. We prove the remaining case $r=10$ when $n\ge3$ by the second-order argument of Subsection~\ref{subsec:second-order-rank-ten}.

    The first-order infinitesimal principle is classical; compare conditions~(1) and~(3) and the proof of \cite[Thm.~3.1]{Voisin_HodgeLoci}. The passage from real tuples to nearby rational Hodge tuples is a local simultaneous form of Green's density criterion, recalled in \cite[Prop.~3.3]{Voisin_HodgeLoci}. Shepherd--Barron's frame makes the required first-order conditions explicit.

    Keep $n\ge2$, $h=n-1$, and $N=10n-2$. Work on a sufficiently small full marked deformation chart $B$ in $\Wsb$, centred at $0$, and write $S_b$ for the surface over $b\in B$. Here a chart is ``full'' if it has dimension $10n-2$ and parametrizes all sufficiently small deformations of the Jacobian elliptic fibration with its zero section, with the base considered up to isomorphism. Use the Schiffer tangent basis, the distinct finite ramification coordinates $z_1,\ldots,z_N$, and the basis $\eta_1,\ldots,\eta_N$ of $H^{1,1}_{\mathrm{prim}}(S_0)$ constructed above. Extend $\omega_0,\ldots,\omega_{h-1}$ to a holomorphic frame of $F^2$ over $B$. For a flat real class $\gamma\in\Lambda_{n,\Rb}$, define
    \begin{equation}\label{eq:er1-period}
    \Phi_\gamma(b)=\bigl(\gamma\cdot\omega_0(b),\ldots,\gamma\cdot\omega_{h-1}(b)\bigr)
    \qquad
    Z_\gamma=\Phi_\gamma^{-1}(0)
    \end{equation}
    Since $\gamma$ is real, $Z_\gamma$ is precisely its type-$(1,1)$ locus.

    Write $A_{\gamma,b}=\diff\Phi_\gamma(b)\colon T_bB\to\Cb^h$. If $\gamma$ is Hodge at $0$, Lemma~\ref{lem:ci-evaluation} identifies $A_{\gamma,0}$ in the Schiffer tangent basis as
    \begin{equation}\label{eq:er1-jacobian}
    \begin{gathered}
    A_{\gamma,0}
    =
    V\cdot\operatorname{diag}\bigl(c_i(\omega_0)(\gamma\cdot\eta_i)\bigr)_{i=1}^N\\
    V=(z_i^j)_{\substack{0\le j<h\\1\le i\le N}}
    \end{gathered}
    \end{equation}
    We keep the non-zero factors $c_i(\omega_0)$ from the preceding construction.

    We begin with one imposed class $\alpha$. If $\alpha$ is Hodge at $0$ and pairs non-trivially with every $\eta_i$, then \eqref{eq:er1-jacobian} shows that $A_{\alpha,0}$ is surjective and $T_0Z_\alpha=\ker A_{\alpha,0}$. For an additional integral class $\beta$ that is Hodge at $0$, it suffices to find a tangent vector $v\in\ker A_{\alpha,0}$ with $A_{\beta,0}(v)\ne0$. This preserves $\alpha$ to first order while changing the Hodge type of $\beta$, so $\beta$ cannot remain Hodge throughout $Z_\alpha$.

    We will prove that $\ker A_{\alpha,0}\subseteq\ker A_{\beta,0}$ can hold only when $\beta$ is proportional to $\alpha$. In \eqref{eq:er1-jacobian}, put $a_i=c_i(\omega_0)(\alpha\cdot\eta_i)$ and $b_i=c_i(\omega_0)(\beta\cdot\eta_i)$. The following lemma shows that the kernel inclusion forces $b=ca$, which translates back into proportionality of the classes.

    \begin{lem}\label{lem:er1-kernel}
    Let $z_1,\ldots,z_N$ be distinct complex numbers, let $h\ge1$, and assume $N\ge h$ and $N>2h-2$. Set $V=(z_i^j)_{0\le j<h,\,1\le i\le N}$. For $a\in(\Cb^*)^N$ and $b\in\Cb^N$, put $A=V\cdot\operatorname{diag}(a)$ and $C=V\cdot\operatorname{diag}(b)$. If $\ker A\subseteq\ker C$, then $b=ca$ for some $c\in\Cb$.
    \end{lem}

    \begin{proof}
    The first $h$ columns of $V$ have non-zero determinant $\prod_{1\le i<k\le h}(z_k-z_i)$, so $A$ has row rank $h$. Every row of $C$ annihilates $\ker A$, hence belongs to the row space of $A$. A vector in that row space has entries $a_i p(z_i)$ for a polynomial $p$ of degree at most $h-1$. Applying this to the first row of $C$ gives
    \begin{equation}\label{eq:er1-ratio}
    b_i=a_i p(z_i)
    \qquad(1\le i\le N)
    \end{equation}
    For $h=1$, this proves the assertion.

    For $h\ge2$, apply the same argument to the last row of $C$. There is a polynomial $q$ of degree at most $h-1$ satisfying
    \[
    z_i^{h-1}b_i=a_i q(z_i)
    \qquad(1\le i\le N)
    \]
    Dividing by $a_i\ne0$ and using \eqref{eq:er1-ratio}, we find that $z^{h-1}p(z)-q(z)$ vanishes at all $N$ points $z_i$. Its degree is at most $2h-2<N$, so it vanishes identically. If $p\ne0$, the identity $z^{h-1}p=q$ gives $h-1+\deg p=\deg q\le h-1$, so $p$ is constant. If $p=0$, then $b=0$. In either case $b=ca$.
    \end{proof}

    We now apply this lemma to period derivatives. The non-vanishing condition on $\alpha$ will ensure that, for every $\beta\in\Lambda_n\setminus\Qb\alpha$ of type $(1,1)$ at $0$, there is a tangent vector $v\in T_0Z_\alpha$ such that $A_{\beta,0}(v)\ne0$.

    \begin{prop}\label{prop:er1-generic}
    Let $\alpha\in\Lambda_n$ be non-zero, indivisible in $\Lambda_n$, and of type $(1,1)$ on $S_0$. Assume
    \begin{equation}\label{eq:er1-non-vanishing}
    \alpha\cdot\eta_i\ne0
    \qquad(1\le i\le N)
    \end{equation}
    Then $Z_\alpha$ is smooth of complex codimension $h=n-1$ at $0$. For every $\beta\in\Lambda_n\setminus\Qb\alpha$ that is Hodge at $0$, one has
    \begin{equation}\label{eq:er1-first-order}
    \left.A_{\beta,0}\right|_{T_0Z_\alpha}\ne0
    \end{equation}
    After replacing $Z_\alpha$ by a sufficiently small connected representative of its germ, there is a countable union $\Ec$ of proper closed analytic subsets of $Z_\alpha$ such that
    \begin{equation}\label{eq:er1-hodge-lattice}
    \Lambda_n\cap H^{1,1}(S_b)=\bZ\alpha
    \qquad
    \rk\MW(S_b/\Pb^1)=1
    \qquad(b\in Z_\alpha\setminus\Ec)
    \end{equation}
    These points are dense in $Z_\alpha$ and have full measure for every measure absolutely continuous with respect to intrinsic smooth volume on $Z_\alpha$.
    \end{prop}

    \begin{proof}
    Condition \eqref{eq:er1-non-vanishing} allows us to apply Theorem~\ref{thm:KS-transversality} to $\alpha$. Its Jacobian calculation shows that $A_{\alpha,0}$ is surjective, so $Z_\alpha$ is smooth of codimension $h$ and
    \begin{equation}\label{eq:er1-tangent}
    T_0Z_\alpha=\ker A_{\alpha,0}
    \end{equation}
    Shrink the chart so that $Z_\alpha$ is a connected complex manifold.

    Let $\beta\in\Lambda_n\setminus\Qb\alpha$ be Hodge at $0$. If \eqref{eq:er1-first-order} failed, then $\ker A_{\alpha,0}\subseteq\ker A_{\beta,0}$. Apply Lemma~\ref{lem:er1-kernel} with $a_i=c_i(\omega_0)(\alpha\cdot\eta_i)$ and $b_i=c_i(\omega_0)(\beta\cdot\eta_i)$. The numerical hypotheses hold since $N\ge h$ and $N-(2h-2)=8n+2>0$. The lemma gives $b_i=ca_i$ for some $c\in\Cb$. Cancelling $c_i(\omega_0)\ne0$, we obtain
    \[
    (\beta-c\alpha)\cdot\eta_i=0
    \qquad(1\le i\le N)
    \]
    Both $\alpha$ and $\beta$ lie in $H^{1,1}_{\mathrm{prim}}(S_0)$. The intersection pairing on this space is non-degenerate: its real form is negative definite by the Hodge index theorem. Since the $\eta_i$ form a basis, the displayed equalities force $\beta=c\alpha$ in $\Lambda_{n,\Cb}$. Comparing an integral lattice coordinate in which $\alpha$ is non-zero gives $c\in\Qb$, a contradiction. This proves \eqref{eq:er1-first-order}.

    For each $\beta\in\Lambda_n\setminus\Qb\alpha$, the subset $Z_\alpha\cap Z_\beta$ is closed analytic in $Z_\alpha$. If $\beta$ is not Hodge at $0$, this subset does not contain $0$ and is therefore proper. If $\beta$ is Hodge at $0$, equation \eqref{eq:er1-first-order} shows that $\Phi_\beta$ cannot vanish identically on $Z_\alpha$, so the subset is again proper. Thus
    \[
    \Ec=
    \bigcup_{\beta\in\Lambda_n\setminus\Qb\alpha}
    (Z_\alpha\cap Z_\beta)
    \]
    is a countable union of proper closed analytic subsets. Each is nowhere dense and has measure zero with respect to every smooth positive density on $Z_\alpha$. Baire's theorem gives density of the complement, and countable additivity gives the measure assertion.

    For $b\notin\Ec$, every integral Hodge class in $\Lambda_n$ belongs to $\Qb\alpha$. Since $\alpha$ is indivisible, $\Lambda_n\cap\Qb\alpha=\bZ\alpha$. The integral Shioda identification established in the proof of Theorem~\ref{thm:KS-higher-rank}, valid for every surface in $\Wsb$, therefore gives $\MW(S_b/\Pb^1)\simeq\Lambda_n\cap H^{1,1}(S_b)=\bZ\alpha$. This proves \eqref{eq:er1-hodge-lattice}.
    \end{proof}

    To obtain density, we must produce surfaces carrying a class that satisfies the preceding proposition. We first choose a real Hodge class with all the required pairings non-zero. We then approximate it by rational classes in the full primitive real cohomology and use the implicit function theorem to make those classes Hodge on nearby surfaces.

    \begin{cor}\label{cor:er1-density}
    For every $n\ge2$, the locus
    \[
    \bigl\{
    [S]\in\Wsb:
    \rk\MW(S/\Pb^1)=1
    \bigr\}
    \]
    is dense in $\cW_{\min,n}$ for the complex analytic topology.
    \end{cor}

    \begin{proof}
    Let $U\subset\cW_{\min,n}$ be a non-empty analytic open subset. By Lemma~\ref{lem:sb-dense-minimal}, choose $[S_0]\in U\cap\Wsb$ and a full marked deformation chart $B$ centred at $0$ whose image lies in $U\cap\Wsb$.

    Each functional $\lambda\mapsto\lambda\cdot\eta_i$ on $H^{1,1}_{\mathrm{prim}}(S_0,\Rb)$ is non-zero, by non-degeneracy of the intersection pairing and $\eta_i\ne0$. Avoiding their finitely many proper kernels, choose
    \begin{equation}\label{eq:er1-real-choice}
    \lambda\in H^{1,1}_{\mathrm{prim}}(S_0,\Rb)
    \qquad
    \lambda\cdot\eta_i\ne0
    \quad(1\le i\le N)
    \end{equation}
    By \eqref{eq:er1-jacobian}, the derivative in the $B$ variables of
    \begin{equation}\label{eq:er1-parameter-map}
    \begin{gathered}
    \Psi_1\colon\Lambda_{n,\Cb}\times B\longrightarrow\Cb^h\\
    \Psi_1(v,b)=\bigl(v\cdot\omega_j(b)\bigr)_{0\le j<h}
    \end{gathered}
    \end{equation}
    is surjective at $(\lambda,0)$. Applying the holomorphic implicit function theorem on an $h$-coordinate slice, as in the proof of Theorem~\ref{thm:local-simultaneous-realization}, gives a holomorphic map $\tau_1$ near $\lambda$ with $\tau_1(\lambda)=0$ and $\Psi_1(v,\tau_1(v))=0$. Choose non-zero vectors $v_m\in\Lambda_n\otimes_{\bZ}\Qb$ converging to $\lambda$ in the domain of $\tau_1$, and put $b_m=\tau_1(v_m)$. Then $b_m\to0$ and $v_m$ is Hodge on $S_{b_m}$.

    Write $F^1$ for the holomorphic Hodge filtration subbundle with fibre $H^{2,0}(S_b)\oplus H^{1,1}_{\mathrm{prim}}(S_b)$. By \cite[Prop.~5.1]{SB_Torelli}, the holomorphic quotient $F^1/F^2$ is the direct image of a line bundle on the relative ramification divisor. Over the ordered chart $B$, this divisor splits into $N$ copies of $B$, giving line subbundles of $F^1/F^2$. The smooth Hodge splitting identifies this quotient with the varying spaces $H^{1,1}_{\mathrm{prim}}(S_b)$ and these line subbundles with the Schiffer image lines. After shrinking $B$, choose continuous non-zero representatives $\eta_i(b)$ with $\eta_i(0)=\eta_i$. Hence
    \[
    v_m\cdot\eta_i(b_m)\longrightarrow\lambda\cdot\eta_i\ne0
    \]
    For sufficiently large $m$, an indivisible generator $\alpha_m$ of $\Lambda_n\cap\Qb v_m$ therefore satisfies the hypotheses of Proposition~\ref{prop:er1-generic} at $b_m$. The Hodge locus $Z_{\alpha_m}$ contains points of exact Mordell--Weil rank $1$ arbitrarily close to $b_m$, whose images lie in $U$.
    \end{proof}

    For several imposed classes, we first state the geometric criterion for generic exact rank and then construct tuples satisfying it.

    Extend $\Phi_\gamma$ and $A_{\gamma,b}$ complex linearly to flat classes $\gamma\in\Lambda_{n,\Cb}$. For a tuple $\boldsymbol\alpha=(\alpha_1,\ldots,\alpha_r)$, put
    \[
    \begin{aligned}
    \Phi_{\boldsymbol\alpha}&=(\Phi_{\alpha_1},\ldots,\Phi_{\alpha_r})\\
    Z_{\boldsymbol\alpha}&=\bigcap_{a=1}^r Z_{\alpha_a}\\
    \mathcal A_{\boldsymbol\alpha,b}&=\diff\Phi_{\boldsymbol\alpha}(b)
    \end{aligned}
    \]

    When the classes $\alpha_a$ are Hodge at $0$ and $\mathcal A_{\boldsymbol\alpha,0}$ has rank $rh$, its kernel is the tangent space to $Z_{\boldsymbol\alpha}$ at $0$. The kernel equality below says that the only primitive $(1,1)$-classes whose period derivatives vanish in every such tangent direction are the linear combinations of the imposed classes. For integral classes, this is enough to exclude every additional class outside their rational span.

    \begin{prop}\label{prop:er-generic}
    Let $\alpha_1,\ldots,\alpha_r\in\Lambda_n$ be linearly independent classes of type $(1,1)$ on $S_0$, and put $L_{\Qb}=\operatorname{span}_{\Qb}\{\alpha_1,\ldots,\alpha_r\}$ and $L=\Lambda_n\cap L_{\Qb}$. Suppose that $\mathcal A_{\boldsymbol\alpha,0}$ has rank $rh$ and that
    \begin{equation}\label{eq:er-rigidity}
    \left\{
    \beta\in H^{1,1}_{\mathrm{prim}}(S_0):
    \left.A_{\beta,0}\right|_{\ker\mathcal A_{\boldsymbol\alpha,0}}=0
    \right\}
    =
    \operatorname{span}_{\Cb}\{\alpha_1,\ldots,\alpha_r\}
    \end{equation}
    Then $Z_{\boldsymbol\alpha}$ is smooth of complex codimension $rh$ at $0$. For every $\beta\in\Lambda_n\setminus L_{\Qb}$ that is Hodge at $0$, one has
    \[
    \left.A_{\beta,0}\right|_{T_0Z_{\boldsymbol\alpha}}\ne0
    \]
    After replacing $Z_{\boldsymbol\alpha}$ by a sufficiently small connected representative of its germ, there is a countable union $\Ec$ of proper closed analytic subsets such that
    \[
    \Lambda_n\cap H^{1,1}(S_b)=L
    \qquad
    \rk\MW(S_b/\Pb^1)=r
    \qquad(b\in Z_{\boldsymbol\alpha}\setminus\Ec)
    \]
    The complement of $\Ec$ is dense and has full measure for every measure absolutely continuous with respect to smooth volume on $Z_{\boldsymbol\alpha}$.
    \end{prop}

    \begin{proof}
    The rank hypothesis and the holomorphic implicit function theorem give smoothness and $T_0Z_{\boldsymbol\alpha}=\ker\mathcal A_{\boldsymbol\alpha,0}$. Shrink to a connected representative. If an integral class $\beta$ is Hodge at $0$ and its period derivative vanishes on this tangent space, condition \eqref{eq:er-rigidity} gives $\beta=\sum_{a=1}^r c_a\alpha_a$ with $c_a\in\Cb$. Choose an integral basis of $\Lambda_n$ and an invertible $r\times r$ minor of the coordinate matrix with columns $\alpha_1,\ldots,\alpha_r$. Solving the corresponding equations for the $c_a$ shows that every $c_a$ is rational. Hence $\beta\in L_{\Qb}$, proving the first-order assertion.

    For every $\beta\in\Lambda_n\setminus L_{\Qb}$, the subset $Z_{\boldsymbol\alpha}\cap Z_\beta$ is closed analytic and proper. If $\beta$ is not Hodge at $0$, this subset does not contain $0$; otherwise properness follows from the first-order assertion. Set
    \[
    \Ec=
    \bigcup_{\beta\in\Lambda_n\setminus L_{\Qb}}
    \bigl(Z_{\boldsymbol\alpha}\cap Z_\beta\bigr)
    \]
    As in the proof of Proposition~\ref{prop:er1-generic}, countability of the lattice, Baire's theorem, and the measure-zero property of proper analytic subsets give the density and measure assertions.

    All the $\alpha_a$ remain Hodge on $Z_{\boldsymbol\alpha}$, so every element of $L$ is Hodge there. Outside $\Ec$, no integral class outside $L_{\Qb}$ is Hodge. Thus the primitive integral Hodge lattice is exactly $L$. The lattice $L$ has rank $r$ and is saturated in $\Lambda_n$, meaning that $\Lambda_n/L$ is torsion-free. The integral Shioda identification therefore gives the asserted Mordell--Weil rank.
    \end{proof}

    It remains to find tuples satisfying \eqref{eq:er-rigidity}. As in the $1$-class case, we first work with complex coefficient vectors. Expressing the required rank conditions by non-vanishing minors will then allow us to pass to real tuples.

    \begin{lem}\label{lem:er-simultaneous-kernel}
    Let $z_1,\ldots,z_N$ be distinct complex numbers, and let $h,r\ge1$ satisfy $rh\le N$. Assume that either $h=1$ or $r(h+1)\le N$. Put $V=(z_i^j)_{0\le j<h,\,1\le i\le N}$. For $x\in\Cb^N$, define $B_x=V\cdot\operatorname{diag}(x)$. Given $\boldsymbol x=(x^1,\ldots,x^r)\in(\Cb^N)^r$, write $x^a=(x_1^a,\ldots,x_N^a)$, so that $(B_{x^a})_{j,i}=z_i^j x_i^a$. Stack these $h\times N$ matrices and set
    \[
    \mathcal A_{\boldsymbol x}
    =
    \begin{pmatrix}
    B_{x^1}\\
    \vdots\\
    B_{x^r}
    \end{pmatrix}
    \qquad
    \mathcal K_{\boldsymbol x}
    =
    \left\{
    y\in\Cb^N:
    \ker\mathcal A_{\boldsymbol x}\subseteq\ker B_y
    \right\}
    \]
    There is a non-empty Zariski open subset $\mathcal U\subset(\Cb^N)^r$ such that every $\boldsymbol x\in\mathcal U$ satisfies
    \begin{equation}\label{eq:er-simultaneous-kernel}
    \rk\mathcal A_{\boldsymbol x}=rh
    \qquad
    \mathcal K_{\boldsymbol x}
    =
    \operatorname{span}_{\Cb}\{x^1,\ldots,x^r\}
    \end{equation}
    \end{lem}

    \begin{proof}
    If $h=1$, the rows of $\mathcal A_{\boldsymbol x}$ are $x^1,\ldots,x^r$. For every linearly independent tuple, the annihilator of $\ker\mathcal A_{\boldsymbol x}$ is their row space, giving \eqref{eq:er-simultaneous-kernel}. Such tuples form a non-empty Zariski open subset because $r\le N$.

    For exact rank, each block must also detect whether an additional period condition vanishes along the simultaneous Hodge locus. We therefore use blocks of size $h+1$: the extra column leaves a $1$ dimensional kernel, and vanishing on this kernel will force the additional coefficient vector to be proportional to the imposed one on that block.

    Suppose $h\ge2$. Choose disjoint subsets $C_1,\ldots,C_r\subset\{1,\ldots,N\}$ with $|C_a|=h+1$, and let $x^a$ be the indicator vector of $C_a$. After ordering the columns by these blocks, the non-zero blocks of $\mathcal A_{\boldsymbol x}$ are the matrices $V_{C_a}$, each of rank $h$. Hence $\rk\mathcal A_{\boldsymbol x}=rh$.

    For each $a$, the kernel of $V_{C_a}$ is $1$ dimensional. Choose a vector $\theta^a\in\Cb^N$ supported on $C_a$ such that $(\theta_i^a)_{i\in C_a}$ generates this kernel. Every coordinate $\theta_i^a$ with $i\in C_a$ is non-zero: otherwise it would give a non-trivial relation among at most $h$ Vandermonde columns, which are linearly independent. Thus $\theta^a\in\ker\mathcal A_{\boldsymbol x}$.

    Let $y=(y_1,\ldots,y_N)\in\mathcal K_{\boldsymbol x}$. If $i$ lies outside all the blocks, the coordinate vector supported at $i$ belongs to $\ker\mathcal A_{\boldsymbol x}$. Its image under $B_y$ has first coordinate $y_i$, so $y_i=0$. On each block, the condition $B_y\cdot\theta^a=0$ gives
    \[
    V_{C_a}\cdot\bigl(y_i\theta_i^a\bigr)_{i\in C_a}=0
    \]
    Since $(\theta_i^a)_{i\in C_a}$ spans $\ker V_{C_a}$, there is a scalar $c_a$ such that $y_i\theta_i^a=c_a\theta_i^a$ for every $i\in C_a$. These coordinates $\theta_i^a$ are non-zero, so $y_i=c_a$ on $C_a$. Thus $y=\sum_a c_ax^a$. Conversely, every linear combination of the $x^a$ belongs to $\mathcal K_{\boldsymbol x}$, so this tuple satisfies \eqref{eq:er-simultaneous-kernel}.

    It remains to show that the two conditions persist on a Zariski open neighbourhood of this tuple. We express the restriction of $B_y$ to $\ker\mathcal A_{\boldsymbol x}$ in coordinates, so that the required kernel dimension becomes a matrix rank condition. Choose a set $I$ of $rh$ columns for which $(\mathcal A_{\boldsymbol x})_I$ is invertible at the constructed tuple, and put $J=\{1,\ldots,N\}\setminus I$. On the open set where this minor remains invertible, define the linear map $R_{\boldsymbol x}\colon\Cb^N\to\operatorname{Mat}_{h,N-rh}(\Cb)$ by
    \[
    R_{\boldsymbol x}(y)
    =
    (B_y)_J-(B_y)_I
    \cdot(\mathcal A_{\boldsymbol x})_I^{-1}
    \cdot(\mathcal A_{\boldsymbol x})_J
    \]
    The kernel of $\mathcal A_{\boldsymbol x}$ consists of the vectors whose coordinates satisfy
    \[
    u_I
    =
    -(\mathcal A_{\boldsymbol x})_I^{-1}
    \cdot(\mathcal A_{\boldsymbol x})_J
    \cdot u_J
    \]
    Thus $B_y\cdot u=R_{\boldsymbol x}(y)\cdot u_J$ for every $u\in\ker\mathcal A_{\boldsymbol x}$, and consequently $\mathcal K_{\boldsymbol x}=\ker R_{\boldsymbol x}$. Full row rank of $\mathcal A_{\boldsymbol x}$ implies that the $x^a$ are linearly independent. Their span always lies in $\ker R_{\boldsymbol x}$, so $\rk R_{\boldsymbol x}\le N-r$. At the constructed tuple, equality holds.

    After choosing bases for the source and target of the linear map $R_{\boldsymbol x}$, its coefficient matrix has entries that are rational functions of $\boldsymbol x$, with denominators given by powers of the chosen invertible minor. Non-vanishing of a suitable $(N-r)\times(N-r)$ minor of this coefficient matrix therefore defines a non-empty Zariski open neighbourhood on which $\rk R_{\boldsymbol x}=N-r$. This proves the lemma.
    \end{proof}

    We now prove Theorem~\ref{thm:er-density}. In the first-order cases, we pull the open set of coefficient tuples back to the real Hodge space and choose a tuple there. We approximate this tuple by rational classes in the full real primitive cohomology and deform the surface so that the approximations become Hodge. We also verify that the kernel condition survives this deformation. The rank $10$ case uses the corresponding second-order construction from the next subsection.

    \begin{proof}[Proof of Theorem~\ref{thm:er-density}]
    Let $U\subset\cW_{\min,n}$ be non-empty and analytically open, and let $H>0$. By Lemma~\ref{lem:sb-dense-minimal}, choose a relatively compact marked deformation chart with image in $U\cap\Wsb$. Lemma~\ref{lem:bounded-height-avoidance} gives a non-empty open subset on which no non-zero integral Hodge class $\gamma$ in primitive cohomology has $-\gamma^2/2\le H$. Choose a sufficiently small full marked deformation chart $B$ inside this subset, centred at $0$, and use the notation above.

    Suppose first that $n\ge3$ and $r=10$. Apply Proposition~\ref{prop:second-order-rank-ten} inside this chart. It gives an $8$ dimensional simultaneous fixed class Hodge germ defined by 10 independent rational classes, with generic Mordell--Weil rank $10$. Replace each rational class by an indivisible integral generator of its rational line. This changes neither the common Hodge locus nor the second-jet kernel condition. After shrinking to a connected representative $Z\subset B$, the chosen height avoidance condition gives $\widehat h(P)>H$ for every non-zero section on every surface over $Z$. This proves the assertion in this case.

    In the remaining cases, either $n=2$ or $r\le9$. We use the first-order criterion. Put $W=H^{1,1}_{\mathrm{prim}}(S_0,\Rb)$. The map
    \[
    \kappa\colon W\otimes_{\Rb}\Cb\longrightarrow\Cb^N
    \qquad
    \kappa(v)=\bigl(c_i(\omega_0)(v\cdot\eta_i)\bigr)_{i=1}^N
    \]
    is an isomorphism because the intersection pairing is non-degenerate, the $\eta_i$ form a basis, and every $c_i(\omega_0)$ is non-zero.

    If $n=2$, then $h=1$ and $r\le N=18$. If $n\ge3$, then
    \[
    r(h+1)=rn\le9n\le10n-2=N
    \]
    Lemma~\ref{lem:er-simultaneous-kernel} therefore applies. Pulling its non-empty Zariski open subset back through $\kappa$, applied to each entry of the tuple, gives a non-empty Zariski open subset of $(W\otimes_{\Rb}\Cb)^r$. Since the real form $W^r$ is Zariski dense, choose a real tuple $\boldsymbol\lambda=(\lambda_1,\ldots,\lambda_r)\in W^r$ whose image satisfies \eqref{eq:er-simultaneous-kernel}. The tuple is linearly independent. By \eqref{eq:er1-jacobian}, its simultaneous period Jacobian has rank $rh$, and the kernel condition translates into
    \[
    \left\{
    \beta\in H^{1,1}_{\mathrm{prim}}(S_0):
    \left.A_{\beta,0}\right|_{\ker\mathcal A_{\boldsymbol\lambda,0}}=0
    \right\}
    =
    \operatorname{span}_{\Cb}\{\lambda_1,\ldots,\lambda_r\}
    \]

    As in the local realization argument, we approximate $\boldsymbol\lambda$ by rational classes in the full cohomology space and deform the surface so that these approximations become Hodge. We must also retain the kernel condition, which will exclude additional persistent Hodge classes outside the span of the imposed classes. Consider
    \[
    \begin{gathered}
    \Psi_r\colon(\Lambda_{n,\Cb})^r\times B\longrightarrow\Cb^{rh}\\
    \Psi_r(\boldsymbol v,b)
    =
    \bigl(v_a\cdot\omega_j(b)\bigr)_{\substack{1\le a\le r\\0\le j<h}}
    \end{gathered}
    \]
    Its derivative in the $B$ variables is surjective at $(\boldsymbol\lambda,0)$. Applying the holomorphic implicit function theorem on an $rh$-coordinate slice, as in the proof of Theorem~\ref{thm:local-simultaneous-realization}, gives a holomorphic map $\tau_r$ from a neighbourhood of $\boldsymbol\lambda$ to $B$, with $\tau_r(\boldsymbol\lambda)=0$ and $\Psi_r(\boldsymbol v,\tau_r(\boldsymbol v))=0$.

    Choose rational tuples $\boldsymbol v_m=(v_{1,m},\ldots,v_{r,m})\in(\Lambda_n\otimes_{\bZ}\Qb)^r$ converging to $\boldsymbol\lambda$ in this neighbourhood, and put $b_m=\tau_r(\boldsymbol v_m)$. Then $b_m\to0$, and every $v_{a,m}$ is Hodge on $S_{b_m}$ because it is real and its period pairings vanish. For sufficiently large $m$, the tuple remains linearly independent and $\mathcal A_{\boldsymbol v_m,b_m}$ has rank $rh$.

    We check that the kernel condition also persists. For tuples $\boldsymbol v$ of Hodge classes at $b$ with $\rk\mathcal A_{\boldsymbol v,b}=rh$, consider the complex linear map
    \[
    \begin{gathered}
    T_{\boldsymbol v,b}\colon H^{1,1}_{\mathrm{prim}}(S_b)
    \longrightarrow
    \operatorname{Hom}_{\Cb}\!\left(\ker\mathcal A_{\boldsymbol v,b},\Cb^h\right)\\
    T_{\boldsymbol v,b}(\beta)
    =
    \left.A_{\beta,b}\right|_{\ker\mathcal A_{\boldsymbol v,b}}
    \end{gathered}
    \]
    The kernel of this map consists of the primitive $(1,1)$-classes whose period equations vanish to first order along the simultaneous Hodge locus. It always contains the imposed classes; we must show that no additional kernel directions appear under a small deformation.

    At $(\boldsymbol\lambda,0)$, the map $T_{\boldsymbol\lambda,0}$ has rank $N-r$. The smooth Hodge splitting gives local continuous frames for its domain spaces, while an invertible $rh$-column minor of $\mathcal A_{\boldsymbol v,b}$ gives a continuously varying basis of $\ker\mathcal A_{\boldsymbol v,b}$. Thus the matrix entries of $T_{\boldsymbol v,b}$ vary continuously. A minor establishing rank $N-r$ at $(\boldsymbol\lambda,0)$ remains non-zero at $(\boldsymbol v_m,b_m)$ for sufficiently large $m$, giving $\rk T_{\boldsymbol v_m,b_m}\ge N-r$. On the other hand, the $r$ independent classes $v_{a,m}$ lie in its kernel, giving $\rk T_{\boldsymbol v_m,b_m}\le N-r$. Consequently,
    \[
    \ker T_{\boldsymbol v_m,b_m}
    =
    \operatorname{span}_{\Cb}\{v_{1,m},\ldots,v_{r,m}\}
    \]

    For each $a$, let $\alpha_{a,m}$ be an indivisible generator of $\Lambda_n\cap\Qb v_{a,m}$. Rescaling the individual classes preserves the simultaneous Hodge locus, the rank of its period Jacobian, and the kernel condition. Proposition~\ref{prop:er-generic}, applied at $b_m$, therefore gives a smooth connected local simultaneous Hodge locus $Z$ of dimension $N-rh$ on which the Mordell--Weil rank is exactly $r$ outside a countable union of proper closed analytic subsets.

    Shrink $Z$ inside $B$. By the choice of $B$ and the integral Shioda identification, every non-zero section on every surface parametrized by $Z$ has canonical height greater than $H$. The complement of the exceptional analytic subsets is non-empty and maps into $U$, proving exact rank $r$ density as well as the stronger assertion.
    \end{proof}

    For $n\ge3$, the range $r\le9$ obtained by the first-order part of the proof of Theorem~\ref{thm:er-density} is the dimensional limit of the complex first-order criterion \eqref{eq:er-rigidity}. Indeed, the restriction map $T_{\boldsymbol v,b}$ in the preceding proof must have rank $N-r$, while its target has dimension $h(N-rh)$. Thus, for $h>1$,
    \[
    N-r\le h(N-rh)
    \quad\Longleftrightarrow\quad
    r(h+1)\le N
    \]
    Since $h=n-1$ and $N=10n-2$, this is equivalent to $r\le9$. Lemma~\ref{lem:er-simultaneous-kernel} attains this range.

    \subsection{Second-order period equations and exact rank 10}
    \label{subsec:second-order-rank-ten}

    For $n\ge3$, the complex first-order criterion~\eqref{eq:er-rigidity} cannot isolate the span of the $10$ imposed classes. We shall use second derivatives to distinguish the additional primitive $(1,1)$-classes. If a rational class remained of type $(1,1)$ throughout the germ, its period functions would vanish identically. It is therefore enough to arrange that their values and first and second derivatives vanish simultaneously only for linear combinations of the $10$ chosen classes.

    Retain $h=n-1$ and $N=10n-2=10h+8$. On an $N$ dimensional marked deformation chart $B$, fix a holomorphic frame $\omega_0,\ldots,\omega_{h-1}$ of $F^2$, the bundle whose fibre at $b$ is $H^{2,0}(S_b)$. A class $\beta$ is held constant by the marking, and we write
    \[
    \Phi_\beta(b)=\bigl(\beta\cdot\omega_0(b),\ldots,\beta\cdot\omega_{h-1}(b)\bigr).
    \]
    Thus $\Phi_\beta$ is an $h$-tuple of holomorphic functions. For real $\beta$, its common zero locus is exactly the locus where $\beta$ has type $(1,1)$.

    We shall use second derivatives of these functions on a smooth submanifold $Z\subset B$. If $s_1,\ldots,s_d$ are local coordinates on $Z$ centred at $b$, the condition
    \[
    j_b^2(\Phi_\beta|_Z)=0
    \]
    means that the value and all first and second partial derivatives vanish:
    \[
    \Phi_\beta(b)=0,\qquad
    \frac{\partial(\Phi_\beta|_Z)}{\partial s_i}(b)=0,\qquad
    \frac{\partial^2(\Phi_\beta|_Z)}{\partial s_i\partial s_k}(b)=0
    \quad(1\le i,k\le d).
    \]
    The simultaneous vanishing of these quantities is independent of the coordinates and of the chosen frame of $F^2$. The notation $j_b^2$ denotes the Taylor expansion through degree two. For a holomorphic map vanishing at $b$, its order of vanishing is the least total degree of a non-zero Taylor term among its components, with the convention that the identically zero germ has infinite order.

    \begin{prop}\label{prop:second-order-rank-ten}
    Let $n\ge3$. For every non-empty analytic open subset $U\subset\Wsb$, there are a marked deformation chart $B\to\Wsb$ and a connected smooth $8$ dimensional simultaneous fixed class Hodge locus $Z\subset B$, with image in $U$, cut out transversely by the period equations of $10$ independent classes $\alpha_1,\ldots,\alpha_{10}\in\Lambda_{n,\Qb}$, such that
    \begin{equation}\label{eq:second-order-kernel}
    \bigl\{\beta\in H^{1,1}_{\mathrm{prim}}(S_b):j_b^2(\Phi_\beta|_Z)=0\bigr\}=L_{\Cb}
    \qquad(b\in Z),
    \end{equation}
    where $L_{\Qb}=\operatorname{span}_{\Qb}\{\alpha_1,\ldots,\alpha_{10}\}$ and $L_{\Cb}=L_{\Qb}\otimes_{\Qb}\Cb$. Outside a countable union of proper closed analytic subsets of $Z$, the Mordell--Weil rank is exactly $10$. If $\beta\in\Lambda_{n,\Qb}\setminus L_{\Qb}$ is Hodge at $b\in Z$, then $\Phi_\beta|_Z$ has order of vanishing $1$ or $2$ at $b$.
    \end{prop}

    The proof has two parts. First, we make one explicit calculation at a Fermat surface. We choose ten complex classes whose simultaneous period equations have independent differentials, determine all classes whose first derivatives vanish along their common zero locus, and use three families of second derivatives to exclude the remaining classes. Second, we pass from this calculation to the rational classes required in the proposition. Polynomial non-vanishing gives real tuples, a connected real analytic parameter space gives such tuples over points in every open subset, and rational approximation together with the implicit function theorem makes nearby rational classes Hodge. We begin with the derivative formula needed for the calculation.

    Let $\alpha_1,\ldots,\alpha_{10}\in H^{1,1}_{\mathrm{prim}}(S_b)$ be complex classes for which
    \[
    J=d\Phi_{\boldsymbol\alpha,b}\colon T_bB\longrightarrow\Cb^{10h},
    \qquad
    \Phi_{\boldsymbol\alpha}=(\Phi_{\alpha_1},\ldots,\Phi_{\alpha_{10}}),
    \]
    is surjective. Their common zero locus germ $Z$ is smooth, and
    \[
    \dim Z=N-10h=8,\qquad K=T_bZ=\ker J.
    \]
    The classes are allowed to be complex here. Their period equations define a holomorphic zero locus, but we do not assert that these classes remain of type $(1,1)$ away from $b$.

    To differentiate along $Z$, choose local coordinates on $B$ centred at $b$ and take a holomorphic curve $c(t)$ in $Z$ with $c(0)=b$ and $c'(0)=\xi\in K$. The ordinary chain rule gives
    \[
    \left.\frac{d^2}{dt^2}\Phi_\beta(c(t))\right|_{t=0}
    =d^2\Phi_{\beta,b}(\xi,\xi)+d\Phi_{\beta,b}\bigl(c''(0)\bigr).
    \]
    The second term is necessary because a curve that remains in $Z$ generally has non-zero acceleration in the surrounding space $B$. Differentiating the constraint $\Phi_{\boldsymbol\alpha}(c(t))=0$ twice gives
    \[
    J\bigl(c''(0)\bigr)=-d^2\Phi_{\boldsymbol\alpha,b}(\xi,\xi).
    \]
    Choose a right inverse $R\colon\Cb^{10h}\to T_bB$, so that $J\circ R=\mathrm{id}$. Then $R(y)$ solves $J(v)=y$, and any two solutions differ by an element of $K=\ker J$. Therefore
    \[
    c''(0)\equiv
    -R\bigl(d^2\Phi_{\boldsymbol\alpha,b}(\xi,\xi)\bigr)
    \pmod K.
    \]

    Suppose now that $\beta\in H^{1,1}_{\mathrm{prim}}(S_b)$ and that $d\Phi_{\beta,b}|_K=0$. The component of $c''(0)$ in $K$ then makes no contribution to the chain rule. We substitute this expression into the chain rule and recover the mixed second derivatives from the diagonal ones, obtaining the Hessian of $\Phi_\beta|_Z$:
    \begin{equation}\label{eq:second-order-restricted-Hessian}
    \mathcal Q_\beta(\xi,\zeta)
    =d^2\Phi_{\beta,b}(\xi,\zeta)
    -\bigl(d\Phi_{\beta,b}\circ R\bigr)
    \bigl(d^2\Phi_{\boldsymbol\alpha,b}(\xi,\zeta)\bigr),
    \qquad \xi,\zeta\in K.
    \end{equation}
    The first term is the second derivative in $B$; the second accounts for the constraint that the curve $c(t)$ lies in $Z$. Two right inverses differ by a map with image in $K$, so the expression is independent of $R$. Since $\Phi_\beta(b)=0$ and $d(\Phi_\beta|_Z)_b=0$, this Hessian is intrinsic, with values in $(F_b^2)^*$.

    We first establish the required kernel identity at a Fermat surface, whose diagonal symmetries give explicit cohomology classes and allow the period derivatives to be calculated from monomials.

    \begin{lem}\label{lem:second-order-Fermat-witness}
    Let $n\ge3$, and let $S_0$ be the minimal resolution of the hypersurface
    \[
    F_0=Y^2+X^3+u^D+v^D=0,
    \qquad D=6n,
    \]
    in the coarse weighted projective space $\mathbb P(1,1,2n,3n)$ with coordinates $[u:v:X:Y]$. There exist an $N$ dimensional marked local deformation chart $(B,0)$ centred at $S_0$ and $10$ linearly independent complex classes $\alpha_1,\ldots,\alpha_{10}\in H^{1,1}_{\mathrm{prim}}(S_0)$ such that
    \[
    J=d\Phi_{\boldsymbol\alpha,0}\colon T_0B\longrightarrow\Cb^{10h}
    \]
    is surjective. Their common zero locus germ $(Z,0)$ is smooth of dimension $8$ and satisfies
    \[
    \bigl\{\beta\in H^{1,1}_{\mathrm{prim}}(S_0):j_0^2(\Phi_\beta|_Z)=0\bigr\}
    =\operatorname{span}_{\Cb}\{\alpha_1,\ldots,\alpha_{10}\}.
    \]
    \end{lem}

    \begin{proof}
    We use the weighted hypersurface to compute period derivatives; the surfaces in the deformation family are its smooth minimal resolutions. Set
    \[
    B_a=u^av^{D-a}\quad(2\le a\le D-2),
    \qquad
    A_a=Xu^av^{4n-a}\quad(0\le a\le4n).
    \]
    A coefficient direction $B_a$, for example, means the first-order deformation $F_0+\varepsilon B_a$. Consider the linear family
    \[
    \mathcal F=F_0+\sum_{a=2}^{D-2}s_aB_a+\sum_{a=0}^{4n}t_aA_a.
    \]
    The coefficients of $u^D,u^{D-1}v,uv^{D-1},v^D$ remain fixed. At $F_0$, the infinitesimal Weierstrass scaling and changes of coordinates on the base span precisely these four omitted coefficient directions. Thus the displayed family is a local deformation slice of dimension
    \[
    (D-3)+(4n+1)=10n-2=N.
    \]

    The affine cone of $F_0=0$ is smooth away from its vertex, so the weighted hypersurface is quasi-smooth. Its only coarse singular point lies on $u=v=0$. In the chart $X=1$, choose one of the two non-zero values of $Y$. Since $\partial F_0/\partial Y\ne0$, the implicit function theorem eliminates $Y$. The stabilizer of this choice is $\mu_n$, acting on $(u,v)$ by scalar multiplication, so the singularity is of type $\frac1n(1,1)$.

    The same elimination works with the coefficients as parameters. Locally, the family is the product of this quotient singularity with the coefficient slice. Resolving the quotient factor gives a simultaneous resolution, whose exceptional curve is the zero section and has self-intersection $-n$. After shrinking to a simply-connected neighbourhood and choosing a marking, we obtain the smooth family over $(B,0)$. Its central surface has height $n$ and $6n$ fibres of type $\II$, since $u^{6n}+v^{6n}$ has simple zeros. In particular, $S_0$ lies outside $\Wsb$. We shall move to $\Wsb$ after establishing the required non-vanishing.

    The residue description gives a polynomial model for primitive cohomology. Writing $F$ and $O$ for the fibre and zero section classes on the resolution, the pullback $H$ of the rational weighted hyperplane class has the form $F+cO$ with $H\cdot O=0$; since $F\cdot O=1$ and $O^2=-n$, we have $H=F+\frac1nO$. Orthogonality to $H$ and $O$ is therefore equivalent to orthogonality to $F$ and $O$. Hence primitive residue cohomology identifies with $\langle F,O\rangle^\perp$ on the resolution. We use the classical Griffiths--Steenbrink residue description for quasi-smooth weighted hypersurfaces, as recalled in \cite[\S8, p.~310 and Lem.~8.2]{Kloosterman_NL}.

    The Jacobian ring is the ring of weighted homogeneous polynomials modulo the partial derivatives of $F_0$:
    \[
    R_{F_0}=\Cb[u,v,X]/(u^{D-1},v^{D-1},X^2).
    \]
    Thus a reduced monomial has $u$- and $v$-exponents at most $D-2$ and $X$-exponent at most one. The relevant graded pieces are
    \[
    \begin{gathered}
    T=(R_{F_0})_D,
    \qquad H^{2,0}=(R_{F_0})_{n-2},\\
    H^{1,1}_{\mathrm{prim}}=(R_{F_0})_{7n-2},
    \qquad H^{0,2}=(R_{F_0})_{13n-2}.
    \end{gathered}
    \]
    Here $T$ is the tangent space to the coefficient slice. The other three pieces give the reduced residue bases of the indicated Hodge spaces.

    Let $\Omega$ be the contraction of $du\wedge dv\wedge dX\wedge dY$ with the weighted Euler vector field $u\partial_u+v\partial_v+2nX\partial_X+3nY\partial_Y$. Its weighted degree is $5n+2$. Thus homogeneity requires the numerator $g$ in $g\Omega/F_0^p$ to have degree $pD-(5n+2)$; for $p=1,2,3$, this gives the degrees $n-2$, $7n-2$, and $13n-2$ displayed above.

    We write $\operatorname{Res}$ for the cohomological residue of a rational three-form on the complement of the hypersurface. A representative with denominator $\mathcal F^p$ has poles of order at most $p$ along the hypersurface. In particular,
    \[
    \omega_j=\operatorname{Res}\frac{u^jv^{h-1-j}\Omega}{F_0}
    \qquad(0\le j<h)
    \]
    is a basis of $H^{2,0}(S_0)$. The two families of pole-two residues are
    \[
    \begin{aligned}
    \operatorname{Res}\frac{u^av^{7h+5-a}\Omega}{F_0^2}
    &\qquad(h+1\le a\le6h+4),\\
    \operatorname{Res}\frac{Xu^av^{5h+3-a}\Omega}{F_0^2}
    &\qquad(0\le a\le5h+3).
    \end{aligned}
    \]

    Each of these pole-two residues lies a priori in $H^{2,0}(S_0)\oplus H^{1,1}_{\mathrm{prim}}(S_0)$, and at the Fermat surface it has pure type $(1,1)$. To see this, scale the coordinates by the roots of unity preserving $F_0$. A reduced monomial $u^av^bX^c$ has character $(a+1,b+1,c+1,1)$, with entries taken modulo $(D,D,3,2)$; this records the eigenvalues of its residue under these scalings. The characters in the three reduced residue bases are all distinct. Since the scalings preserve the Hodge decomposition, each $1$ dimensional eigenspace lies in a single Hodge summand. Thus the displayed pole-two residues lie in $H^{1,1}_{\mathrm{prim}}(S_0)$ itself. The reduced pole-one and pole-three residues similarly have types $(2,0)$ and $(0,2)$.

    Let $U_a,V_a$ be the classes dual to the two pole-two residue families under the intersection pairing. We use scalar coordinates $x_i,y_i$ for a class in this dual basis:
    \[
    \beta=\sum_{i=h+1}^{6h+4}x_iU_i+\sum_{i=0}^{5h+3}y_iV_i.
    \]
    Pairing $\beta$ with the residue indexed by $a$ in the first or second family therefore gives $x_a$ or $y_a$, respectively.

    For a fixed polynomial $g$ of degree $n-2$, the residues $\omega_g=\operatorname{Res}(g\Omega/\mathcal F)$ extend holomorphically over the family. In the flat cohomology trivialization, differentiation in coefficient directions $G_1,G_2$ gives
    \begin{equation}\label{eq:second-order-residue-derivatives}
    \begin{aligned}
    \partial_{G_1}\omega_g&=-\operatorname{Res}\frac{gG_1\Omega}{F_0^2},\\
    \partial_{G_1}\partial_{G_2}\omega_g&=2\operatorname{Res}\frac{gG_1G_2\Omega}{F_0^3}.
    \end{aligned}
    \end{equation}
    Multiplication by $B_a$ or $A_a$ shifts the $u$-exponent by $a$. Pairing with $\beta$ therefore gives
    \begin{equation}\label{eq:second-order-Fermat-Jacobian}
    \bigl(d\Phi_\beta(B_a)\bigr)_j=-x_{a+j},
    \qquad
    \bigl(d\Phi_\beta(A_a)\bigr)_j=-y_{a+j}
    \qquad(0\le j<h).
    \end{equation}
    Coordinates outside the indicated ranges for $x_i,y_i$ are understood to be zero. Indeed, when a numerator index lies outside these ranges, the corresponding pole-two residue either vanishes or reduces to a holomorphic two-form, which pairs trivially with $\beta\in H^{1,1}_{\mathrm{prim}}(S_0)$.

    For the indices used below, \eqref{eq:second-order-Fermat-Jacobian} shows that the $j$-th period equation for $U_k$ has derivative $-1$ in the direction $B_{k-j}$ and zero in every other $B$-direction. The corresponding statement for $V_k$ uses $A_{k-j}$. We use this rule to choose ten classes whose period equations have independent differentials.

    Choose
    \[
    \begin{aligned}
    (\alpha_1,\ldots,\alpha_5)&=(U_{h+1},U_{2h+1},U_{3h+2},U_{4h+2},U_{5h+3}),\\
    \alpha_6&=U_{6h+4}+U_{3h},\\
    (\alpha_7,\ldots,\alpha_{10})&=(V_h,V_{2h+1},V_{3h+1},V_{4h+3}),
    \end{aligned}
    \]
    and put $L_{\Cb}=\operatorname{span}_{\Cb}\{\alpha_1,\ldots,\alpha_{10}\}$. These are auxiliary complex classes. Throughout this calculation, $Z$ denotes the common zero locus germ of their period equations.

    For each of the first five classes $U_k$, select the $h$ columns
    \[
    B_k,\ B_{k-1},\ \ldots,\ B_{k-h+1}.
    \]
    For $\alpha_6$, make the same selection using its summand $U_{6h+4}$. For each of the four $V_k$-classes, select the corresponding columns $A_k,A_{k-1},\ldots,A_{k-h+1}$. The displayed choices of indices ensure that these $10h$ columns are distinct. If $\alpha_6$ were replaced by $U_{6h+4}$, the selected $10h\times10h$ submatrix would be $-I_{10h}$. Within this submatrix, adding $U_{3h}$ changes only entries in the rows of $\alpha_6$ and columns selected for earlier classes. The submatrix is therefore block triangular with diagonal blocks $-I_h$, so $J$ has rank $10h$.

    We can now find $\ker J$ without listing the complementary column indices. Set
    \[
    q=B_{5h+6}-B_{2h+2},\qquad r=B_{4h+3},\qquad w=A_0.
    \]
    The two monomials defining $q$ have the same image under $J$: each contributes only to the equation for $\alpha_6$ with $j=h-2$, where its derivative is $-1$. Their difference therefore lies in $\ker J$. Formula~\eqref{eq:second-order-Fermat-Jacobian} also shows that the other seven vectors in
    \[
    q,\ r,\ B_{5h+4},\ w,\ A_{h+1},\ A_{3h+2},\ A_{3h+3},\ A_{4h+4}
    \]
    belong to $\ker J$. These eight vectors are linearly independent. Since $\dim\ker J=N-10h=8$, they form a basis of $K=\ker J$.

    We next determine which classes $\beta=\sum x_iU_i+\sum y_iV_i$ satisfy $d\Phi_\beta|_K=0$. Evaluation on $r$, $B_{5h+4}$, and $q$ gives, respectively,
    \[
    \begin{aligned}
    x_{4h+3+j}&=0,\\
    x_{5h+4+j}&=0,\\
    x_{5h+6+j}&=x_{2h+2+j}
    \end{aligned}
    \qquad(0\le j<h).
    \]
    Taking $j=h-2$ in the last equation gives $x_{6h+4}=x_{3h}$, which is the relation represented by $\alpha_6$. The five $A$-directions similarly impose $y_{a+j}=0$ for each $A_a$ in the displayed basis and every $0\le j<h$.

    Together, these equations leave precisely the coefficients of the ten chosen classes and of three additional families. Put
    \[
    G=\operatorname{span}_{\Cb}\bigl\{U_{h+2+j},\,U_{3h+3+j},\,V_{2h+2+j}:0\le j\le h-2\bigr\}.
    \]
    We have therefore proved
    \begin{equation}\label{eq:second-order-first-kernel}
    \bigl\{\beta\in H^{1,1}_{\mathrm{prim}}(S_0):d\Phi_\beta|_K=0\bigr\}=L_{\Cb}\oplus G.
    \end{equation}
    Each additional family has $h-1$ members, so $\dim G=3(h-1)$. The calculation also applies when $h=2$, in which case each family consists of a single class.

    It remains to show that the second derivatives distinguish the classes in $G$. Write
    \[
    \beta=\sum_{j=0}^{h-2}\bigl(a_jU_{h+2+j}+b_jU_{3h+3+j}+c_jV_{2h+2+j}\bigr).
    \]
    We shall use the tangent pairs $(q,q)$, $(r,q)$, and $(q,w)$ to determine $a_j$, $b_j$, and $c_j$, respectively. All derivatives below are evaluated at the Fermat surface.

    We first calculate the second derivatives in the surrounding coefficient space. Formula~\eqref{eq:second-order-residue-derivatives} expresses them as residues with denominator $F_0^3$. To pair these residues with $\beta$, we reduce their pole order using the partial derivatives of $F_0$. In the second derivative, reducing a $u$-exponent $e\ge D$ lowers it by $D$ and gives the coefficient $(e-D+1)/D$; the same rule applies to $v$. This coefficient includes the factor $2$ from differentiation. An exponent equal to $D-1$ gives zero. Terms that reduce to pole order one, or remain reduced at pole order three, have type $(2,0)$ or $(0,2)$ and therefore pair trivially with $\beta$.

    For example, put $g_j=u^jv^{h-1-j}$. For the pair $(q,w)$, multiplication gives
    \[
    g_jqw=Xu^{5h+6+j}v^{6h+3-j}-Xu^{2h+2+j}v^{9h+7-j}.
    \]
    For $0\le j\le h-2$, the first term is already reduced at pole order three, so its pairing with $\beta$ is zero. In the second term, reducing the $v$-exponent gives the coefficient
    \[
    \frac{(9h+7-j)-D+1}{D}=\frac{3h+2-j}{D}.
    \]
    The resulting pole-two residue is indexed by $2h+2+j$ in the $V$-family, so its pairing with $\beta$ is $c_j$. Including the minus sign gives the third identity below. Applying the same reduction to the other two tangent pairs gives
    \[
    \begin{aligned}
    D\bigl(d^2\Phi_\beta(q,q)\bigr)_j&=-2(h+3+j)a_j,\\
    D\bigl(d^2\Phi_\beta(r,q)\bigr)_j&=(3h+4+j)b_j,\\
    D\bigl(d^2\Phi_\beta(q,w)\bigr)_j&=-(3h+2-j)c_j
    \end{aligned}
    \qquad(0\le j\le h-2).
    \]
    For $(q,q)$, only the cross term contributes to the pairing with $\beta\in G$; the two square terms give pole-two residues whose indices lie outside the support of $\beta$. For $(r,q)$, the product involving $B_{2h+2}$ either vanishes or reduces twice to a pole-one residue, so it also contributes zero.

    These are derivatives in the coefficient space. To obtain derivatives along $Z$, we must also satisfy the imposed period equations to second order. For each tangent pair $(\xi,\zeta)$, choose a vector $\nu$ such that
    \[
    J(\nu)=d^2\Phi_{\boldsymbol\alpha}(\xi,\zeta).
    \]
    Formula~\eqref{eq:second-order-restricted-Hessian} then gives
    \[
    \mathcal Q_\beta(\xi,\zeta)
    =d^2\Phi_\beta(\xi,\zeta)-d\Phi_\beta(\nu).
    \]
    Indeed, $\nu-R\bigl(d^2\Phi_{\boldsymbol\alpha}(\xi,\zeta)\bigr)$ belongs to $K$, and $d\Phi_\beta$ vanishes on $K$. We therefore need only find a suitable $\nu$ for each of the three pairs.

    For $(q,q)$, the cross term contributes $-(4h+4)/D$ to the last period equation of $\alpha_2$. Since $B_{h+2}$ has first derivative $-1$ in that equation and zero in the other imposed equations, this contribution is accounted for by the component $(4h+4)B_{h+2}/D$ of $\nu$. The square terms affect only the equations of $\alpha_5$ with $j=h-1$ and, when $h\ge3$, $j=h-3$. These contributions can be accounted for using $B_{4h+4}$ and $B_{4h+6}$, respectively. Both directions have zero first period derivative against every $\beta\in G$, so these additional components do not contribute to $d\Phi_\beta(\nu)$. When $h=2$, only $B_{4h+4}$ is needed.

    For $(r,q)$, the only non-zero imposed second derivative is $(4h+3)/D$ in the last period equation of $\alpha_4$. Formula~\eqref{eq:second-order-Fermat-Jacobian} therefore allows us to take
    \[
    \nu=-\frac{4h+3}{D}B_{3h+3}.
    \]
    For $(q,w)$, the only non-zero imposed second derivative is $-(2h+3)/D$ in the last period equation of $\alpha_9$. This follows from the same monomial calculation above with $j=h-1$: the first term has $u$-exponent $D-1$ and vanishes, while the second gives the residue dual to $\alpha_9=V_{3h+1}$. We may therefore take
    \[
    \nu=\frac{2h+3}{D}A_{2h+2}.
    \]

    For $\beta\in G$ and $0\le j\le h-2$, formula~\eqref{eq:second-order-Fermat-Jacobian} gives
    \[
    \bigl(d\Phi_\beta(B_{h+2})\bigr)_j=-a_j,\qquad
    \bigl(d\Phi_\beta(B_{3h+3})\bigr)_j=-b_j,\qquad
    \bigl(d\Phi_\beta(A_{2h+2})\bigr)_j=-c_j.
    \]
    Subtracting the correction in each case now yields
    \begin{equation}\label{eq:second-order-diagonal-Hessians}
    \begin{aligned}
    D\bigl(\mathcal Q_\beta(q,q)\bigr)_j
    &=\bigl[-2(h+3+j)+(4h+4)\bigr]a_j
    =2(h-1-j)a_j,\\[0.3em]
    D\bigl(\mathcal Q_\beta(r,q)\bigr)_j
    &=\bigl[(3h+4+j)-(4h+3)\bigr]b_j
    =-(h-1-j)b_j,\\[0.3em]
    D\bigl(\mathcal Q_\beta(q,w)\bigr)_j
    &=\bigl[-(3h+2-j)+(2h+3)\bigr]c_j
    =-(h-1-j)c_j
    \end{aligned}
    \end{equation}
    Since $h-1-j\ge1$, vanishing of these components forces $a_j=b_j=c_j=0$ for every $j$, hence $\beta=0$.

    Finally, let $\beta\in H^{1,1}_{\mathrm{prim}}(S_0)$ have vanishing restricted first and second derivatives. By \eqref{eq:second-order-first-kernel}, write $\beta=\ell+\gamma$ with $\ell\in L_{\Cb}$ and $\gamma\in G$. The period functions of $\ell$ vanish identically on $Z$, so the restricted derivatives of $\beta$ are those of $\gamma$. The preceding calculation gives $\gamma=0$, hence $\beta\in L_{\Cb}$. Conversely, every class in $L_{\Cb}$ has identically zero period functions on $Z$. This proves the lemma.
    \end{proof}

    \begin{proof}[Proof of Proposition~\ref{prop:second-order-rank-ten}]
    We now pass from the explicit complex tuple in the lemma to rational Hodge classes over points in every prescribed open subset. The two conditions to preserve are surjectivity of the imposed period derivative $J$ and the assertion that the common kernel of the restricted first and second derivatives is precisely the span of the ten classes.

    To express the second condition by a matrix, fix a tuple $\boldsymbol\alpha\in H^{1,1}_{\mathrm{prim}}(S_b)^{10}$ for which $J$ is surjective. Choose a non-zero $10h$-column minor of $J$ and use the remaining eight coordinates to parametrize the common zero locus germ:
    \[
    \iota\colon(\Cb^8,0)\longrightarrow(Z,b)\subset(B,b).
    \]
    Choose a basis of $H^{1,1}_{\mathrm{prim}}(S_b)$. For each basis class $\beta$, extended near $b$ by flat transport, record all first and second partial derivatives of $\Phi_\beta\circ\iota$ at $0$ as a column of a matrix $M(b,\boldsymbol\alpha)$. This is the matrix of a linear map in $\beta$. Since $\Phi_\beta(b)=0$ for every class in its domain, its kernel is exactly
    \[
    \bigl\{\beta\in H^{1,1}_{\mathrm{prim}}(S_b):j_b^2(\Phi_\beta|_Z)=0\bigr\}.
    \]
    The ten independent classes $\alpha_a$ lie in this kernel. Its domain has dimension $N$, so the desired kernel identity is equivalent to
    \[
    \rk J=10h,\qquad \rk M(b,\boldsymbol\alpha)=N-10.
    \]
    The ranks do not depend on the choices of coordinates or bases. We use local choices of this kind in the arguments below.

    First keep the Fermat surface $S_0$ fixed and vary the complex tuple in $H^{1,1}_{\mathrm{prim}}(S_0)^{10}$. The entries of $J$ are linear in the tuple. On the open subset where the chosen minor of $J$ is non-zero, a basis of $\ker J$ and a right inverse of $J$ depend rationally on the tuple. Implicit differentiation expresses the first and second derivatives of $\iota$ using the inverse of the corresponding $10h\times10h$ submatrix and the first two derivatives of the period equations. Hence these derivatives, and the entries of $M(0,\boldsymbol\alpha)$, also depend rationally on the tuple.

    The lemma supplies a tuple for which both ranks have the required values. Select corresponding non-zero minors of $J$ and $M$. Clearing their denominators gives non-zero polynomials in the tuple coordinates. The real form $H^{1,1}_{\mathrm{prim}}(S_0,\Rb)^{10}$ is Zariski dense in its complexification, so these polynomials are simultaneously non-zero at some real tuple. Thus both rank conditions hold for a real tuple at $S_0$.

    We may now vary the surface as well. Extend the chosen real tuple in a local real analytic frame of $H^{1,1}_{\mathrm{prim}}(S_b,\Rb)$. At each point, compute $J$ and $M$ using flat extensions of the values of the tuple at that point. Their matrix entries vary continuously, so the selected minors remain non-zero nearby. Since the local deformation chart maps openly to moduli and $\Wsb$ is dense, we obtain a simply-branched surface and a real tuple satisfying both rank conditions. We next show that such points occur in every analytic open subset.

    Let $T^{\mathrm{sb}}\subset V_n$ be the open coefficient locus parametrizing $\Wsb$. It is a non-empty complex Zariski open subset of a vector space, and hence is connected. Over it, let $\mathcal T$ be the space of complex ten-tuples in primitive $(1,1)$-cohomology for which $J$ is surjective. Each fibre is a non-empty Zariski open subset of a complex vector space by Lemma~\ref{lem:simultaneous-SB-minor}, so each fibre is connected. The projection $\mathcal T\to T^{\mathrm{sb}}$ is open, since $\mathcal T$ is open in the corresponding real analytic vector bundle. It follows that $\mathcal T$ is connected: a separation of the total space into two non-empty open subsets would, by connectedness of the fibres and openness of the projection, separate the base.

    The matrices can be computed on local deformation slices. Indeed, the four additional coefficient directions come from Weierstrass scaling and changes of coordinates on $\Pb^1$. They give isomorphic elliptic fibrations, so the marked period map is locally constant along the corresponding orbits. Locally, the coefficient family is the product of a deformation slice with such an orbit. The period functions may therefore be chosen to be pulled back from the slice, and adjoining the orbit coordinates does not change which classes have zero restricted derivatives through order two.

    In local frames over $\mathcal T$, the entries of $M(b,\boldsymbol\alpha)$ are real analytic. Its kernel always contains the ten chosen classes, so its rank is at most $N-10$. Let $\mathcal E\subset\mathcal T$ be the subset where the rank is less than $N-10$. It is locally the common zero locus of the $(N-10)\times(N-10)$ minors, hence is a real analytic subset. The tuple obtained above shows that $\mathcal E$ is proper. A proper real analytic subset of a connected real analytic manifold has empty interior, so $\mathcal E$ contains no non-empty open subset of $\mathcal T$.

    If a non-empty analytic open subset of $T^{\mathrm{sb}}$ admitted no tuple with $\rk M=N-10$, its inverse image would be a non-empty open subset of $\mathcal T$ contained in $\mathcal E$, a contradiction. The base points admitting such a complex tuple are therefore dense. They form an open set as well, because the projection is open and the rank condition is open. At each such base point, the same polynomial argument on a fixed surface replaces the complex tuple by a real tuple with both rank conditions.

    Apply the density statement to the inverse image of the prescribed open subset $U$ in $T^{\mathrm{sb}}$. Choose a point there with a real tuple satisfying both rank conditions, and a marked deformation chart $B$ whose image lies in $U$. Approximate the tuple by rational classes in the full real primitive cohomology, using the marking. The approximating classes need not be Hodge on the original surface; we restore their Hodge type by deforming the surface.

    Choose the $10h$ deformation coordinates corresponding to the selected non-zero minor of $J$, and fix the remaining eight coordinates. The implicit function theorem with the classes as parameters, applied as in Theorem~\ref{thm:local-simultaneous-realization}, solves all the imposed period equations for every sufficiently close rational tuple. This gives a nearby surface on which the rational classes have zero period pairings. They are real, so they have type $(1,1)$ there.

    The selected minors of $J$ and $M$ remain non-zero for sufficiently close approximations. The latter matrix still has the ten independent chosen classes in its kernel, so its rank remains exactly $N-10$. Denote the resulting rational classes by $\alpha_1,\ldots,\alpha_{10}$. Their common Hodge germ is cut out transversely and has dimension eight. After shrinking to a connected representative $Z$, both minors remain non-zero at every point of $Z$, and its image remains in $U$. This proves \eqref{eq:second-order-kernel} throughout $Z$.

    Put $L_{\Qb}=\operatorname{span}_{\Qb}\{\alpha_1,\ldots,\alpha_{10}\}$ and $L_{\Cb}=L_{\Qb}\otimes_{\Qb}\Cb$. For each $\beta\in\Lambda_{n,\Qb}\setminus L_{\Qb}$, the subset $Z\cap\Phi_\beta^{-1}(0)$ is closed analytic in $Z$. It is proper. Otherwise $\Phi_\beta|_Z$ would vanish identically, so $\beta$ would be Hodge at the base point and all its restricted derivatives would vanish there. Equation~\eqref{eq:second-order-kernel} would give $\beta\in L_{\Cb}\cap\Lambda_{n,\Qb}=L_{\Qb}$, a contradiction.

    There are countably many such rational classes, and each of their proper closed analytic subsets is nowhere dense in $Z$. Baire's theorem gives a dense complement. At every point of this complement, the primitive rational Hodge space is exactly $L_{\Qb}$. Its integral Hodge lattice is therefore $\Lambda_n\cap L_{\Qb}$, of rank ten. The integral Shioda identification gives Mordell--Weil rank ten.

    Finally, suppose that $\beta\in\Lambda_{n,\Qb}\setminus L_{\Qb}$ is Hodge at $b\in Z$. The equality $L_{\Cb}\cap\Lambda_{n,\Qb}=L_{\Qb}$ implies that $\beta\notin L_{\Cb}$. Thus $\Phi_\beta(b)=0$, while \eqref{eq:second-order-kernel} gives $j_b^2(\Phi_\beta|_Z)\ne0$. Some first or second partial derivative is therefore non-zero. Equivalently, the holomorphic map $\Phi_\beta|_Z$ has order of vanishing $1$ or $2$ at $b$.
    \end{proof}

    Together with Lemma~\ref{lem:bounded-height-avoidance}, the proposition proves the rank $10$ case of Theorem~\ref{thm:er-density}.

    \medskip

    For a family $Z$ as in Proposition~\ref{prop:second-order-rank-ten}, put $L_{\Qb}=\operatorname{span}_{\Qb}\{\alpha_1,\ldots,\alpha_{10}\}$. The rank jumping locus of Definition~\ref{def:rank-jumping-loci} has the Hodge theoretic description
    \[
    \mathcal E_Z
    =
    \bigcup_{\beta\in\Lambda_n\setminus L_{\Qb}}\bigl(Z\cap\Phi_\beta^{-1}(0)\bigr)
    \]
    The proof shows that each member of this countable union is a proper closed analytic subset of $Z$. Thus $\mathcal E_Z$ has empty interior and measure zero on $Z$, and its complement is precisely the rank $10$ locus. The proposition does not establish whether $\mathcal E_Z$ is non-empty.

    \subsection{Section moduli and comparison with Hodge loci}
    \label{subsec:two-pointed-moduli}

    In this subsection, we interpret the realization theorem in terms of moduli of elliptic surfaces with chosen sections. We first express the canonical height of a section as the degree of a line bundle pulled back from $\overline{\cM}_{1,2}$. We then compare deformations of the sections with the corresponding period conditions.

    Write $\cY=\overline{\cM}_{1,1}\simeq\cP(4,6)$ and $\cX=\overline{\cM}_{1,2}$. Inchiostro identifies $\cX$ with the $(4,6)$-weighted blowup of $\cP(2,3,4)$ at $[1:1:0]$ \cite[Thms.~1.1 and~2.10]{InchiostroM12}. Let
    \[
    \phi\colon\cX\longrightarrow\cY
    \qquad
    \pi\colon\cX\longrightarrow\cP(2,3,4)
    \]
    be the universal stable one-pointed curve and the weighted blowdown, respectively. All weighted projective spaces here are stacks. The exceptional divisor $D$ parametrizes curves whose two markings lie on a rational tail; under the universal curve identification, it is the zero section. We write $\lambda=\cO_{\cY}(1)$, $\cL_{1,2}=\phi^*\lambda$, and $\xi=\pi^*\cO_{\cP(2,3,4)}(1)$. Throughout this subsection, we work on the all-$\I_1$ locus and write $f\colon S\to\Pb^1$ for the elliptic surface fibration.

    \begin{prop}\label{prop:two-pointed-height}
    The line bundles on $\cX$ satisfy
    \begin{equation}\label{eq:two-pointed-tautological}
    \xi\simeq\cL_{1,2}\otimes\cO_{\cX}(D)
    \qquad
    \omega_{\cX}\simeq\cL_{1,2}^{\otimes(-9)}
    \end{equation}
    Let $f\colon S\to\Pb^1$ represent a point of $\cW_n^{\I_1}$, with classifying map $\varphi_S$ and fundamental line bundle $\cL$. A section $P\ne O$ determines a lift $f_P\colon\Pb^1\to\cX$ of $\varphi_S$. If $C_P=P^*O$ is the divisor where $P$ meets $O$, then
    \begin{equation}\label{eq:two-pointed-pullbacks}
    f_P^*\cL_{1,2}\simeq\cL
    \qquad
    f_P^*D=C_P
    \qquad
    (\pi\circ f_P)^*\cO(1)\simeq\cL(C_P)
    \end{equation}
    In particular,
    \begin{equation}\label{eq:two-pointed-height-degree}
    \deg f_P^*\cL_{1,2}=n
    \qquad
    \widehat h(P)=\deg f_P^*\xi=n+(P\cdot O)\ge n
    \end{equation}
    \end{prop}

    \begin{proof}
    We first compute the line bundles in \eqref{eq:two-pointed-tautological}. In weighted coordinates $[x:y:A]$ on $\cP(2,3,4)$, the rational map to $\cY$ is given by $A$ and $y^2-x^3-Ax$, of weights $4$ and $6$. Their pullbacks along $\pi$ have exceptional factors of orders $4$ and $6$. Dividing by these factors gives sections of the fourth and sixth powers of $\xi\otimes\cO_{\cX}(-D)$. These sections, together with this underlying line bundle, define $\phi$. Hence $\cL_{1,2}\simeq\xi\otimes\cO_{\cX}(-D)$ in the integral Picard group of $\cX$.

    The canonical bundle of $\cP(2,3,4)$ is $\cO(-9)$, and the $(4,6)$-weighted blowup has discrepancy $4+6-1=9$. Hence
    \[
    \omega_{\cX}\simeq\xi^{\otimes(-9)}\otimes\cO_{\cX}(9D)\simeq\cL_{1,2}^{\otimes(-9)}
    \]
    which proves the second identity.

    We now construct $f_P$. Since all singular fibres have type $\I_1$, the surface $S$ is the pullback of the universal stable one-pointed curve along $\varphi_S$. The section $P$ therefore determines a lift to $\cX$. A section cannot pass through a fibre node: in the local equation $xy=t$, the conditions $x(0)=y(0)=0$ would imply $\operatorname{ord}_0(xy)\ge2$.

    At a point where $P$ meets $O$, the two-pointed stabilization inserts a rational tail carrying the two markings. Since $D$ is the universal zero section, its pullback is $P^*O$ as an effective Cartier divisor, including intersection multiplicities. Together with $\varphi_S^*\lambda\simeq\cL$ and the first identity in \eqref{eq:two-pointed-tautological}, this proves \eqref{eq:two-pointed-pullbacks}. The lift $f_P$ maps into $\cM_{1,2}$ precisely where the fibre is smooth and the two sections are distinct.

    It remains to identify the degree with the canonical height. Put $d=n+(P\cdot O)$. The class $\gamma_P=[P]-[O]-d[F]$ is orthogonal to $F$ because $(P-O)\cdot F=0$ and $F^2=0$. The coefficient $d$ is forced by orthogonality to $O$, since $(P-O)\cdot O=n+P\cdot O=d$ and $F\cdot O=1$. Since $P^2=O^2=-n$ and $P\cdot F=O\cdot F=1$, we have $\gamma_P^2=-2d$. All fibres are irreducible, so $\gamma_P$ is Shioda's projection of $P$. Our normalization gives $\widehat h(P)=-\gamma_P^2/2=d$, while \eqref{eq:two-pointed-pullbacks} gives $\deg f_P^*\xi=d$. Finally, $P\ne O$ implies $P\cdot O\ge0$, proving \eqref{eq:two-pointed-height-degree}.
    \end{proof}

    We can also describe the map to $\cP(2,3,4)$ directly from the Weierstrass equation. Write $P=(x_P,y_P)$ in $y^2=x^3+a_4x+a_6$, let $s$ be the canonical section of $\cO(C_P)$, and put $\cL_P=\cL(C_P)$. Then
    \[
    \begin{aligned}
    X&=s^2x_P\in H^0(\cL_P^{\otimes2})\\
    Y&=s^3y_P\in H^0(\cL_P^{\otimes3})\\
    A'&=s^4a_4\in H^0(\cL_P^{\otimes4})
    \end{aligned}
    \]
    satisfy $Y^2-X^3-A'X=s^6a_6$ and define $\pi\circ f_P=[X:Y:A']$. To check that these sections have no common zero, consider first a point where $P$ meets $O$ with intersection multiplicity $k$. The coordinates $x_P$ and $y_P$ have poles of orders $2k$ and $3k$, so $X$ and $Y$ are non-zero there and the image is $[1:1:0]$. Away from $C_P$, a common zero of $X,Y,A'$ would imply $a_4=a_6=0$, contrary to the all-$\I_1$ hypothesis. Thus $(\cL_P,X,Y,A')$ is a basepoint-free $(2,3,4)$-weighted linear series whose underlying line bundle has degree $\deg\cL_P=n+(P\cdot O)=\widehat h(P)$, in the terminology of \cite[Def.~2.9 and Prop.~2.10]{BPS}.

    The surface and its chosen section fit into the following commutative diagram:
    \[
    \begin{tikzcd}[column sep=large, row sep=large]
    S \arrow[r] \arrow[d, "f"']
    &
    \cX=\overline{\cM}_{1,2} \arrow[d, "\phi"] \arrow[dr, "\pi"]
    &
    \\
    \Pb^1 \arrow[u, bend left=45, "P"] \arrow[r, "\varphi_S"'] \arrow[ur, "f_P" description]
    &
    \cY=\overline{\cM}_{1,1}\simeq\cP(4,6)
    &
    \cP(2,3,4) \arrow[l, dashed, "\psi"']
    \end{tikzcd}
    \]
    The dashed arrow is the rational map $\psi([x:y:A])=[A:y^2-x^3-Ax]$, whose indeterminacy at $[1:1:0]$ is resolved by the $(4,6)$-weighted blowup $\pi$. Thus the right triangle commutes wherever $\psi$ is defined. The left square is Cartesian, and $f_P$ is the composite of the section $P$ with the upper horizontal map. The morphism $\pi$ contracts the zero section divisor $D\simeq\cY$, whose pullback $f_P^*D=C_P$ records the intersections of $P$ with $O$. The degrees of $f_P^*\cL_{1,2}$ and $f_P^*\xi$ record the Faltings height of $S$ and the canonical height of $P$, respectively.

    We now construct the stack $\cS_{n,\mathbf d}^{\I_1}$ introduced in the Introduction. Its objects are height $n$ all-$\I_1$ Jacobian elliptic surfaces with ordered non-zero sections $P_1,\ldots,P_r$ of canonical heights $d_1,\ldots,d_r$, with the base considered up to isomorphism. The sections determine a map to the fibre product
    \[
    \cX^{(r)}=\underbrace{\cX\times_{\cY}\cdots\times_{\cY}\cX}_{r\text{ factors}}
    \]
    Conversely, consider a map $g=(f_1,\ldots,f_r)\colon\Pb^1\to\cX^{(r)}$ with common projection $\varphi\colon\Pb^1\to\cY$. We require $\deg\varphi^*\lambda=n$, $\deg f_a^*\xi=d_a$, and the pullback of the nodal boundary along $\varphi$ to be a reduced divisor. Pulling back the universal curve along $\varphi$ then gives the elliptic fibration, and the lifts $f_a$ give its sections. At each node, reducedness of the boundary pullback gives the local equation $xy=t$ for a base parameter $t$, so the pullback surface is smooth with only $\I_1$ singular fibres. The positivity of $d_a$ excludes the zero section, since $\xi|_D\simeq\cO_D$. These constructions commute with base change. Thus $\cS_{n,\mathbf d}^{\I_1}$ is the corresponding open locus of maps of the specified degrees, modulo $\PGL_2$. The fibre product allows the sections $P_a$ to meet.

    Forgetting the additional sections defines a morphism
    \[
    \rho\colon\cS_{n,\mathbf d}^{\I_1}\longrightarrow\cW_n^{\I_1}
    \]
    Its source and target have finite stabilizers; automorphisms are required to preserve the zero section. The degrees $d_a=\deg f_a^*\xi$ are locally constant in families, so fixing $\mathbf d$ does not change the local deformation germ of a tuple.

    \begin{prop}\label{prop:two-pointed-local-comparison}
    The stack $\cS_{n,\mathbf d}^{\I_1}$ is an algebraic Deligne--Mumford stack of finite type over $\Cb$, and the forgetful morphism $\rho$ is representable. Let $(S,P_1,\ldots,P_r)$ be a point of $\cS_{n,\mathbf d}^{\I_1}$, and let $B\to\cW_n^{\I_1}$ be a simply-connected analytic \'etale local deformation chart through $[S]$. Choose a marking and a lift of the tuple over the central point $b_0$. Put $\gamma_a=[P_a]-[O]-d_a[F]$, and define
    \[
    \Phi(b)=\bigl(\gamma_a\cdot\omega_j(b)\bigr)_{a,j}
    \qquad
    1\le a\le r,\quad 0\le j\le n-2
    \]
    where $\omega_0(b),\ldots,\omega_{n-2}(b)$ is a local holomorphic frame of $H^{2,0}(S_b)$. If $d\Phi$ has rank $r(n-1)$ at $b_0$, then $\cS_{n,\mathbf d}^{\I_1}$ is smooth at the tuple, of dimension
    \begin{equation}\label{eq:two-pointed-actual-dimension}
    10n-2-r(n-1)
    \end{equation}
    After base change to $B$, the morphism $\rho$ identifies the germ at the chosen lift with $(\Phi^{-1}(0),b_0)$.
    \end{prop}

    \begin{proof}
    We first establish the algebraic moduli construction. Over a coefficient atlas for $\cW_n^{\I_1}$, the universal surface has a relatively ample divisor $A=3O+cF$ for an integer $c>3n$. For a section $P_a$ of canonical height $d_a$, we have $P_a\cdot O=d_a-n$. Since $A\cdot P_a=3(d_a-n)+c$ and $P_a\simeq\Pb^1$, its Hilbert polynomial is
    \[
    \chi\bigl(\cO_{P_a}(mA)\bigr)
    =
    \bigl(3(d_a-n)+c\bigr)m+1
    \]
    In each relative Hilbert scheme with this fixed polynomial, the locus where the universal closed subscheme projects isomorphically onto $\Pb^1$ is open and parametrizes sections of canonical height $d_a$. The fibred product of these open loci over the coefficient atlas is compatible with base change and descends to $\cW_n^{\I_1}$, showing that $\cS_{n,\mathbf d}^{\I_1}$ is an algebraic Deligne--Mumford stack of finite type and that $\rho$ is representable.

    We next compute the deformations of the corresponding map $g=(f_{P_1},\ldots,f_{P_r})\colon\Pb^1\to\cX^{(r)}$. Each component lies in the smooth locus of $\phi$. This includes points where $P_a$ meets $O$, which correspond to smooth points of the original one-pointed fibre under the universal curve identification. Thus $\cX^{(r)}$ is smooth in a neighbourhood of $g(\Pb^1)$.

    Since $P_a\simeq\Pb^1$, $K_S=(n-2)F$, and $F\cdot P_a=1$, adjunction gives $P_a^2=-2-K_S\cdot P_a=-n$. Thus $N_{P_a/S}\simeq\cO_S(P_a)|_{P_a}\simeq\cO_{\Pb^1}(-n)$. Along a section, the vertical tangent line of $f\colon S\to\Pb^1$ identifies with the section's normal bundle. Since $S$ is pulled back from the universal curve, base change therefore identifies $f_{P_a}^*T_{\cX/\cY}$ with $N_{P_a/S}$. Together with $\varphi_S^*T_{\cY}\simeq\cL^{\otimes10}\simeq\cO_{\Pb^1}(10n)$, the relative tangent sequence gives, with $E=g^*T_{\cX^{(r)}}$,
    \begin{equation}\label{eq:two-pointed-tangent-sequence}
    0\longrightarrow\cO_{\Pb^1}(-n)^{\oplus r}
    \longrightarrow E
    \longrightarrow\cO_{\Pb^1}(10n)
    \longrightarrow0
    \end{equation}
    Taking cohomology gives
    \begin{equation}\label{eq:two-pointed-obstruction-sequence}
    0\longrightarrow H^0(E)
    \longrightarrow H^0(\cO_{\Pb^1}(10n))
    \xrightarrow{\delta}\bigoplus_{a=1}^r H^1(N_{P_a/S})
    \longrightarrow H^1(E)\longrightarrow0
    \end{equation}
    For a first-order deformation of $\varphi_S$, local lifts of $P_a$ exist because $\phi$ is smooth along $f_{P_a}$. Their differences on overlaps form a \v{C}ech cocycle with values in $N_{P_a/S}$. Its class is $\delta_a$; it vanishes precisely when the local lifts glue to a deformation of the section.

    To pass to deformations over an unparameterized base, we quotient by $\PGL_2$. This gives
    \[
    T_{b_0}B\simeq H^0(\varphi_S^*T_{\cY})\big/d\varphi_S\bigl(H^0(T_{\Pb^1})\bigr)
    \]
    The denominator has dimension three, since an infinitesimal source stabilizer would preserve the $12n$ distinct discriminant points and hence vanish. These reparametrizations lift through $dg$, so $\delta$ descends to $T_{b_0}B$.

    We next compare this obstruction map with the period derivative. For a section $P=P_a$, consider
    \[
    0\longrightarrow\cO_S\longrightarrow\cO_S(P)\longrightarrow N_{P/S}\longrightarrow0
    \]
    Since $H^0(N_{P/S})=0$, we have $h^0(\cO_S(P))=1$. Also $H^2(\cO_S(P))=0$ by Serre duality: the divisor $K_S-P$ has intersection $-1$ with the nef fibre class, so it cannot be effective. Riemann--Roch gives
    \[
    \chi(\cO_S(P))=n+\tfrac12P\cdot(P-K_S)
    =n+\tfrac12(-n-(n-2))=1
    \]
    It follows that $H^1(\cO_S(P))=0$. The connecting map, also called the divisor semiregularity map, is therefore an isomorphism
    \begin{equation}\label{eq:two-pointed-semiregularity}
    s_P\colon H^1(N_{P/S})\xrightarrow{\sim}H^2(\cO_S)
    \end{equation}

    For the Cartier divisor $P\subset S$, the semiregularity map of Buchweitz--Flenner agrees with $s_P$: both are Serre dual to the restriction map $H^0(K_S)\to H^0(K_S|_P)$. The Kodaira--Spencer map associates to each tangent vector in $T_{b_0}B$ the corresponding first-order deformation class in $H^1(S,T_S)$. Compatibility with this map identifies the image of the section obstruction under $s_P$ with the infinitesimal Hodge obstruction for $[P]$, taking values in $H^2(\cO_S)\simeq H^{0,2}(S)$; see \cite[Thms.~1.3, 4.5 and~5.5(1)--(2), \S5.3]{BuchweitzFlenner_Semiregularity}. Since $[O]$ and $[F]$ remain Hodge, this is also the obstruction for $\gamma_a$. Thus the composite
    \[
    T_{b_0}B\xrightarrow{\delta}\bigoplus_{a=1}^r H^1(N_{P_a/S})\xrightarrow{\oplus_a s_{P_a}}H^2(\cO_S)^{\oplus r}
    \]
    agrees with $d\Phi$, after applying Serre duality $H^2(\cO_S)\simeq H^0(K_S)^\vee$ and evaluating in the chosen frame, up to a fixed non-zero scalar determined by the cycle class and trace conventions. Since each $s_{P_a}$ is an isomorphism, a first-order deformation lifts all the chosen sections if and only if their classes remain of type $(1,1)$ to first order.

    The rank hypothesis now implies that $\delta$ is surjective. By \eqref{eq:two-pointed-obstruction-sequence}, $H^1(E)=0$. For maps from the fixed curve $\Pb^1$ to a smooth Deligne--Mumford stack, the tangent space is $H^0(E)$ and an obstruction space is $H^1(E)$. Indeed, local lifts across a square-zero extension differ by sections of $E$, and their obstruction to gluing is the resulting \v{C}ech class in $H^1(E)$. Thus the mapping stack is smooth at $g$. Using \eqref{eq:two-pointed-tangent-sequence}, its dimension is
    \[
    h^0(E)=\chi(E)=(10n+1)+r(1-n)
    \]
    Quotienting by $\PGL_2$ subtracts $3$, proving \eqref{eq:two-pointed-actual-dimension}. In particular, for $r=1$ under the rank hypothesis, the actual dimension equals the expected dimension $9n-1$, independently of $d_1$.

    It remains to identify the germs. Since $\rho$ is representable, the analytification of $\cS_{n,\mathbf d}^{\I_1}$, after base change to $B$, is an analytic space. Let $H$ be a sufficiently small smooth analytic neighbourhood of the chosen lift $h_0$ in this base change.

    The marked integral classes of the divisors $P_a$ remain constant in this family. Thus the period functions $\Phi$ vanish on $H$, and, since $H$ is reduced, $\rho$ factors holomorphically through $\Phi^{-1}(0)$. The obstruction sequence and the quotient by source reparametrizations give
    \[
    T_{h_0}H\simeq
    H^0(E)\big/dg\bigl(H^0(T_{\Pb^1})\bigr)
    \xrightarrow{\ d\rho\ }
    \ker\bigl(d\Phi(b_0)\bigr)\subset T_{b_0}B
    \]
    where the displayed arrow is an isomorphism. There are no additional tangent directions from deforming the sections on a fixed surface because $H^0(N_{P_a/S})=0$. Both germs are smooth of dimension $10n-2-r(n-1)$, so the analytic inverse function theorem identifies them. The resulting identification is compatible with the forgetful morphism:
    \[
    \begin{tikzcd}[column sep=large, row sep=large]
    (H,h_0) \arrow[rr, "\sim"] \arrow[dr, "\rho"']
    &&
    (\Phi^{-1}(0),b_0) \arrow[dl, hook]
    \\
    &
    (B,b_0)
    &
    \end{tikzcd}
    \]
    Smoothness here is asserted for $\cS_{n,\mathbf d}^{\I_1}$ and its local charts.
    \end{proof}

    The sections constructed in Theorem~\ref{thm:KS-higher-rank} satisfy the transversality hypothesis of the preceding proposition. We therefore obtain the following moduli theoretic interpretation of the realization theorem.

    \begin{cor}\label{cor:two-pointed-realization}
    Let $n\ge2$ and $r\ge1$ be integers satisfying $r(n-1)\le10n-2$. For every $[S]\in\Wsb$, the surfaces and sections with indivisible Shioda classes of Theorem~\ref{thm:KS-higher-rank} give points $(S_m,P_{1,m},\ldots,P_{r,m})\in\cS_{n,\mathbf d_m}^{\I_1}$ with $d_{a,m}=\widehat h(P_{a,m})\to\infty$ for every $a$. At each such point, the stack is smooth of dimension $10n-2-r(n-1)$, and its germ on the chosen marked chart is the simultaneous fixed class Hodge locus germ. Moreover,
    \[
    \begin{aligned}
    \deg f_{P_{a,m}}^*\cL_{1,2}&=n\\
    \deg(\pi\circ f_{P_{a,m}})^*\cO(1)&=d_{a,m}\\
    (P_{a,m}\cdot O_m)&=d_{a,m}-n
    \end{aligned}
    \]
    Every analytic neighbourhood of every point of $\cW_{\min,n}$ contains infinitely many pairwise non-isomorphic surfaces with these smooth section moduli germs. For any prescribed $H>0$, the surfaces and representatives of their germs may be chosen so that every non-zero Mordell--Weil section on every parametrized surface has canonical height greater than $H$.
    \end{cor}

    \begin{proof}
    Theorem~\ref{thm:KS-higher-rank} gives independent sections whose Shioda classes are indivisible and whose heights tend to infinity; the corresponding simultaneous period equations have independent differentials. Proposition~\ref{prop:two-pointed-local-comparison} gives the smoothness and the identification of germs, while Proposition~\ref{prop:two-pointed-height} gives the degree formulas.

    The final assertion follows from Corollary~\ref{cor:rank-density}, using Proposition~\ref{prop:two-pointed-local-comparison} to identify the resulting fixed class Hodge germs with the section moduli germs.
    \end{proof}

    We now combine the local construction, the height bounds, and the section moduli comparison to complete the proofs of the main realization theorems.

    \begin{proof}[Proof of Theorems~\ref{thm:intro-realization} and~\ref{thm:intro-transverse}]
    The construction in the proof of Theorem~\ref{thm:KS-higher-rank} takes place in a single marked local deformation chart $B$ centred at $b_0$ representing $[S]$. It gives points $b_m\to b_0$ representing pairwise non-isomorphic surfaces, independent sections with indivisible Shioda classes, divergence of the least non-zero section height, and independence of the simultaneous period differentials.

    Proposition~\ref{prop:two-pointed-height} identifies the Shioda classes with $[P_{a,m}]-[O_m]-d_{a,m}[F_m]$, where $d_{a,m}=\widehat h(P_{a,m})$. Proposition~\ref{prop:two-pointed-local-comparison} therefore gives the asserted smoothness, dimension, and identification of germs. Fixing the heights does not change these germs, since the heights are locally constant in families with chosen sections; restricting to the open substack $\Wsb$ likewise does not change them.

    Corollary~\ref{cor:rank-density} gives the assertions in arbitrary analytic neighbourhoods, including the uniform lower bounds for all non-zero section heights on representatives of the Hodge germs. After shrinking each representative, the germ comparison extends the $r$ chosen sections uniquely over it. Their marked Shioda classes remain flat, so their height pairing matrix is constant. Taking $r=10$ gives dimension $10n-2-10(n-1)=8$.
    \end{proof}

    \begin{rmk}\label{rmk:ten-section-canonical}
    The number of additional sections has a natural interpretation in the birational geometry of pointed genus one curves. The open locus in $\cX^{(r)}$ parametrizing smooth elliptic curves with $O,P_1,\ldots,P_r$ pairwise distinct identifies with $\cM_{1,r+1}$. Thus $\cX^{(r)}$ is birational to $\overline{\cM}_{1,r+1}$: the zero section contributes one marking. Writing $\kappa$ for Kodaira dimension, the compactified coarse moduli spaces satisfy
    \[
    \kappa\bigl(\overline M_{1,r+1}\bigr)=
    \begin{cases}
    -\infty&1\le r\le9\\
    0&r=10\\
    1&r\ge11
    \end{cases}
    \]
    by \cite[\S4, Thm.~3 and the preceding discussion]{BiniFontanari}.

    The same threshold appears in the canonical bundle and expected deformation dimension of our fibre product target. Let $q\colon\cX^{(r)}\to\cY$ be the common projection, and let $U\subset\cX^{(r)}$ be the open locus where each coordinate lies in the relative smooth locus of $\phi$. Since $\omega_{\cY}\simeq\lambda^{\otimes(-10)}$ and $\omega_{\cX}\simeq\phi^*\lambda^{\otimes(-9)}$, we have
    \[
    \omega_{\cX/\cY}\simeq\phi^*\lambda
    \qquad
    \omega_U\simeq(q|_U)^*\lambda^{\otimes(r-10)}
    \]
    All section maps above factor through $U$, including where the chosen sections meet one another or the zero section. For $g\colon\Pb^1\to U$ with $\deg(q\circ g)^*\lambda=n$, Riemann--Roch gives the expected dimension modulo $\PGL_2$:
    \[
    \chi(g^*T_U)-3=(10-r)n+r-2=10n-2-r(n-1)
    \]
    This is the actual local dimension at transverse tuples by Proposition~\ref{prop:two-pointed-local-comparison}. For $r=10$, the canonical bundle of $U$ is trivial and the expected dimension is $8=\dim U-3$, independently of $n$. For $r=11$, the expected dimension is $9-n$, so positive-dimensional families at height $n\ge9$ have excess dimension. More generally, non-negativity of the expected dimension $10n-2-r(n-1)$ requires $r(n-1)\le10n-2$, which stabilizes at $r\le10$ for $n\ge10$; beyond this range, every realization is non-transverse.
    \end{rmk}



    \section*{Acknowledgements}

    The author thanks Dori Bejleri for discussions of Mordell--Weil rank jumps, François Greer and Klaus Hulek for careful comments on an earlier draft, and Scott Mullane and Matthew Satriano for helpful conversations.

    \medskip

    \bibliographystyle{alpha}
    \bibliography{main.bib}


    \vspace{+19pt}

    \noindent Jun--Yong Park \enspace -- \enspace \texttt{june.park@sydney.edu.au} \\
    \textsc{School of Mathematics and Statistics, University of Sydney, Australia}

\end{document}